\documentclass[11pt,a4paper]{article}
\usepackage{latexsym,epsfig}
\usepackage{amsmath,amssymb,amsthm}
\usepackage{graphicx}
\usepackage{url,hyperref}
\usepackage{subfig}
\usepackage[margin=1.1in]{geometry}
\usepackage{subfig}
\usepackage{epstopdf}

\usepackage[utf8]{inputenc}
\usepackage{enumitem}

\long\def\ignore#1{}

\newenvironment{packed_enum}{
	
	\begin{enumerate}
		\setlength{\itemsep}{1pt}
		\setlength{\parskip}{0pt}
		\setlength{\parsep}{0pt}
	}{\end{enumerate}}

\newenvironment{packed_item}{
	
	\begin{itemize}
		\setlength{\itemsep}{1pt}
		\setlength{\parskip}{0pt}
		\setlength{\parsep}{0pt}
	}{\end{itemize}}

\graphicspath{{.},{figures/}}

\newtheorem{theorem}{Theorem}
\newtheorem{lemma}{Lemma}[section]

\newtheorem{obs}[lemma]{Observation}
\newtheorem{definition}[lemma]{Definition}
\newtheorem{problem}[theorem]{Problem}
\newtheorem{corollary}[lemma]{Corollary}
\newtheorem{proposition}[lemma]{Proposition}

\def\DT{{\mathcal{DT}}}
\mathchardef\mhyphen="2D 
\newcommand{\twopath}[3]{\ensuremath{#1 \mhyphen #2 \mhyphen #3}}
\newcommand{\threepath}[4]{\ensuremath{#1 \mhyphen #2 \mhyphen #3 \mhyphen #4}}
\def\NE{{{\rm NE}}}
\def\NW{{{\rm NW}}}
\def\SE{{{\rm SE}}}
\def\SW{{{\rm SW}}}
\def\R{\mathbb R}

\begin{document}
	
	\title{Coloring points with respect to squares\footnote{A preliminary version of this paper was presented at the 32nd International Symposium on Computational Geometry (SoCG 2016).}}
	
	\author{
		Eyal Ackerman\thanks{
			Department of Mathematics, Physics, and Computer Science,
			University of Haifa at Oranim,
			Tivon 36006, Israel.
			E-mail: {\tt ackerman@sci.haifa.ac.il}.
			Research partially supported by ERC Advanced Research Grant no 267165 (DISCONV).}
		\and
		Bal\'azs Keszegh\thanks{
			Alfr\'ed R\'enyi Institute of Mathematics,
			Hungarian Academy of Sciences,
			H-1053 Budapest, Hungary.
			E-mail: {\tt keszegh@renyi.hu}.
			Research supported by the National Research, Development and Innovation
				Office -- NKFIH under the grant PD 108406 and K 116769 and by the J\'anos Bolyai Research Scholarship of the Hungarian Academy of Sciences.}
		\and
		Mate Vizer\thanks{Alfr\'ed R\'enyi Institute of Mathematics,
			Hungarian Academy of Sciences,
			H-1053 Budapest, Hungary.
			E-mail: {\tt vizermate@gmail.com}.
			Research supported by the National Research, Development and Innovation
			Office -- NKFIH under the grant SNN 116095.
		}
	}

\maketitle

\begin{abstract}
We consider the problem of $2$-coloring geometric hypergraphs.
Specifically, we show that  there is a constant $m$ such that
any finite set of points in the plane $\mathcal{S} \subset {\R}^2$ can be $2$-colored such that every axis-parallel square
that contains at least $m$ points from $\mathcal{S}$ contains points of both colors.
Our proof is constructive, that is, it provides a polynomial-time algorithm for obtaining such a $2$-coloring.
By affine transformations this result immediately applies also when considering $2$-coloring points with respect to homothets of a fixed parallelogram.

keywords: geometric hypergraph coloring, cover-decomposability, self-coverability

\end{abstract}

\section{Introduction}
\label{sec:Intro}

In this paper we consider the problem of coloring a given set of points in the plane
such that every region from a given set of regions contains a point from each color class.
To state our results, known results and open problems, we need the following definitions and notations.

A \emph{hypergraph} is a pair $(\cal V,\cal E)$ where $\cal V$ is a set and $\cal E$ is a set of subsets of $\cal V$.
The elements of $\cal V$ and $\cal E$ are called \emph{vertices} and \emph{hyperedges}, respectively.
For a hypergraph $H:=(\cal V,\cal E)$, let $H|_{m}:=({\cal V}, \{e \in {\cal E} : |e| \geq m\})$.
A {\em proper coloring} of a hypergraph is a coloring of its vertex set such
that in every hyperedge not all vertices are assigned the same color.
Proper colorability of a hypergraph with two colors is also called \emph{Property~B} in the literature.
A \emph{polychromatic $k$-coloring} of a hypergraph is a coloring of its vertex set with $k$ colors such that
every hyperedge contains at least one vertex from each of the $k$ colors.

Given a family of regions ${\cal F}$ in ${\R}^d$ (e.g., all disks in the plane),
there is a natural way to define two types of finite hypergraphs that are dual to each other.
First, for a finite set of points $\cal S$, let $H^{\cal F}(\cal S)$ denote
the \emph{primal hypergraph} on the vertex set $\cal S$ whose hyperedges are
all subsets of $\cal S$ that can be obtained by intersecting $\cal S$ with a member of $\cal F$.
We say that a finite subfamily ${\cal F}_0 \subseteq {\cal F}$ \emph{realizes} $H^{\cal F}(\cal S)$
if for every hyperedge ${\cal S}' \subseteq {\cal S}$ of $H^{\cal F}(\cal S)$ there is $F' \in \mathcal{F}_0$
such that $F' \cap \mathcal{S} = \mathcal{S}'$.
The \emph{dual hypergraph} $H^*({\cal F}_0)$ is defined with respect to a finite multi-subfamily\footnote{In a multisubfamily we allow taking multiple copies of the same set.} ${\cal F}_0 \subseteq {\cal F}$.
Its vertex set is ${\cal F}_0$ and for each point $p\in \R^d$ it has one hyperedge that consists of exactly those regions in ${\cal F}_0$ that contain $p$.

The general problems we are interested in are the following.

\begin{problem}\label{prob:general} For a given family of regions $\cal F$,
	\begin{enumerate}[label=(\roman*)]
		\item Is there a constant $m$ such that for any finite set of points $\cal S$
				the hypergraph $H^{\cal F}({\cal S})|_{m}$ admits a proper $2$-coloring?
		\item Is there a constant $m^*$ such that for any finite subset ${\cal F}_0 \subseteq \cal F$
				the hypergraph $H^*({\cal F}_0)|_{m^*}$ admits a proper $2$-coloring?
		\item Given a constant $k$, is there a constant $m_k$ such that for any finite set of points $\cal S$
				the hypergraph $H^{\cal F}({\cal S})|_{m_k}$ admits a polychromatic $k$-coloring? If so, is $m_k=O(k)$?
		\item Given a constant $k$, is there a constant $m^*_k$ such that for any finite subset ${\cal F}_0 \subseteq \cal F$
				the hypergraph $H^*({\cal F}_0)|_{m^*_k}$ admits a polychromatic $k$-coloring? If so, is $m^*_k=O(k)$?
	\end{enumerate}
\end{problem}

Examples of families $\cal F$ for which such coloring problems are studied are translates of convex sets~\cite{A08,GV11,P86,PT07,PT10,TT07},
homothets of triangles~\cite{CKMUtri,octantmulti,KP11,selfcover,KP12,octantnine},
axis-parallel rectangles~\cite{Ch12,CPST09,PTT,PaT10} and half-planes~\cite{Khalf,SY12}.
If $\cal F$ is the family of disks in the plane, then these hypergraphs generalize Delaunay graphs.

The main motivation for studying proper and polychromatic colorings of such geometric hypergraphs comes from cover-decomposability problems~\cite{PPT} and conflict-free coloring problems~\cite{Smo13}.
We concentrate on the first connection, as the problems we regard are in direct connection with cover-decomposability problems. We give a short introduction to this topic here, however we recommend the interested reader to consult the survey paper~\cite{PPT} or the webpage \cite{domhome} that contains a summary of results about decomposition of multiple coverings and polychromatic colorings.

Multiple coverings and packings were first studied by Davenport and L.\ Fejes T\'oth almost $50$ years ago.
Since then a wide variety of questions related to coverings and packings has been investigated.
In $1986$ Pach \cite{P86} published the first paper about decomposability problems of multiple coverings.
It turned out that this area is rich of deep and exciting questions, and it has important practical applications as well (e.g., in the area of surveillance systems~\cite{GV11,PPT}).
Following Pach's papers, most of the efforts were concentrated on studying coverings by translates of some given shape.
Recently, several researchers started studying cover-decomposability of homothets of a given shape. 

A family of planar sets is called an \emph{$r$-fold covering} of a region $R$,
if every point of $R$ is contained in at least $r$ members of the family.
A $1$-fold covering is simply called a covering.
A family $\cal F$ of planar sets is called {\em cover-decomposable},
if there is an integer $\ell$ with the property that for any multi-subfamily of $\cal F$
that forms an $\ell$-fold covering of the whole plane can be decomposed into two coverings\footnote{In fact in early papers that considered only coverings by translates, all mentioned variants of cover-decomposability were defined for sets instead of families, where a set is cover-decomposable if the family of its translates is cover-decomposable in the way we define it in this paper.}. A family $\cal F$ of planar sets is called {\em totally-cover-decomposable},
if there is an integer $\ell^{T}$ with the property that for any region $R$, any multi-subfamily of $\cal F$
that forms an $\ell^{T}$-fold covering of $R$ can be decomposed into two coverings.

One can also ask for a decomposition into more than two coverings.
That is, whether there exists an integer $\ell_k$ (resp., $\ell_k^{T}$) such that any multi-subfamily of $\cal F$ that forms an $\ell_k$-fold covering of the whole plane
(resp., an $\ell_k^{T}$-fold covering of $R$, for any region $R$), can be decomposed into $k$ coverings of the plane (resp., of $R$).

If we consider only coverings with finite multi-subfamilies, then we call it the \emph{finite cover-decomposition} problem.
It is easy to see that the finite cover-decomposition problem is equivalent to Problems~\ref{prob:general}(ii) and (iv). 

One of the first observations of Pach was that if $\cal F$ is the family of translates of an open convex set, then finite cover-decomposability implies cover-decomposability. 
Thus for the family of the translates of an open convex set $\ell\le m^*$ and $\ell_k\le m^*_k$ in the notation above.
Pach also observed that if $\cal F$ is the family of translates of some set, then Problems~\ref{prob:general}(i) and (ii) are equivalent and also Problems~\ref{prob:general}(iii) and (iv) are equivalent, i.e., $m^*=m$ and $m^*_k=m_k$.
Thus, for the family of translates of an open convex set $\ell\le m^*=m$ and $\ell_k\le m^*_k=m_k$ and so it is enough to consider the primal hypergraph coloring problem.

Pach conjectured that the family of all translates of any open convex planar set is cover-decomposable~\cite{P80}.
During the years researchers acquired a good understanding of convex planar shapes whose translates are cover-decomposable.
On the positive side, Pach's conjecture was verified for every open convex polygon:
Pach himself proved it for every open centrally symmetric convex polygon~\cite{P86},
then Tardos and T\'oth~\cite{TT07} proved the conjecture for every open triangle,
and finally P\'alv\"olgyi and T\'oth~\cite{PT10} proved it for every open convex polygon.
They also gave a complete characterization of open non-convex polygons whose translates are finite cover-decomposable. 
For open convex polygons it is even known that $m_k=m^*_k=O(k)$ \cite{A08,GV11,PT07}.
However, Pach's conjecture was refuted in~\cite{unsplittable}. Specifically, 
it does not hold for disks and for convex shapes with a smooth boundary. 

Considering the {three dimensional space}, it follows from cover-indecomposability of certain non-convex polygons~\cite{concave}
that every bounded polytope is not cover-decomposable.
Thus, it is not easy to come up with a cover-decomposable set in the space.
An important exception is the
\emph{octant}\footnote{An octant is the set of points $\{(x,y,z)| x<a,y<b,z<c\}$
	in the space for some $a,b$ and $c$.}, whose translates were proved to be cover-decomposable \cite{KP11}. The currently best bounds for octants are $5\le m\le 9$ \cite{octantnine} and
	$m_k=m^*=O(k^{5.09})$ \cite{octantmulti,KP12,octantnine}. It is a challenging open problem whether $m_k=m^*_k=O(k)$ in this case.

For a long time no positive results were known about cover-decomposability and geometric hypergraph coloring problems concerning homothets of a given shape.
For disks, the answer is negative for all parts of Problem~\ref{prob:general}~\cite{unsplittable,indec}.
As a first positive result, the cover-decomposability of octants along with a simple reduction
implied that both the primal and dual hypergraphs with respect to homothets of a triangle are properly $2$-colorable:

\begin{theorem}[\cite{KP11,octantnine}]\label{tri2col} For the family $\cal F$ of all homothets of a given triangle both Problems \ref{prob:general}(i) and \ref{prob:general}(iii) have a positive answer with $m=m^*\le 9$.
\end{theorem}

This result was later used to obtain polychromatic colorings of the primal and dual hypergraphs defined by the family of homothets of a fixed triangle.
For the dual hypergraph, the best bound comes from the corresponding result about octants and so it is $m^*_k=O(k^{5.09})$.
For the primal hypergraph there is a better bound $m_k=O(k^{4.09})$ \cite{selfcover,octantnine}.
An important tool for obtaining these results is the notion of \emph{self-coverability} (see Section~\ref{subsec:self-coverability}), which is also essential for proving our results.
The questions whether $m_k=O(k)$ and $m^*_k=O(k)$ for the homothets of a given triangle are still open. The web-page~\cite{domhome} contains an up-to-date collection of results considering all of these problems and related ones.

For polygons other than triangles, somewhat surprisingly, Kov\'acs~\cite{kovacshomot} recently provided a negative answer for Problems~\ref{prob:general}(ii) and (iv).
Namely, he showed that the homothets of any given convex polygon with at least four sides are not cover-decomposable.
In other words, there is no constant $m^*$ for which the dual hypergraph consisting of hyperedges of size at least $m^*$ is $2$-colorable.
Our main contribution is showing that this is not the case when considering $2$-coloring the primal hypergraph.
Indeed, Problem \ref{prob:general}(i) has a positive answer for homothets of any given \emph{parallelogram}.

\begin{theorem}
	\label{thm:parallelogram}
	There is an absolute constant $m_q \leq 215$ such that the following holds.
	Given an (open or closed) parallelogram $Q$ and a finite set of points in the plane $\mathcal{S} \subset \R^2$, 
	the points of $\mathcal{S}$ can be $2$-colored in polynomial time, such that any homothet of $Q$ that contains
	at least $m_q$ points contains points of both colors.
\end{theorem}

This is the first example that exhibits such different behavior for coloring the primal and dual hypergraphs with respect to the family of some geometric regions.
Furthermore, combined with results about self-coverability, the proof of Theorem~\ref{thm:parallelogram} immediately implies the following generalization to polychromatic $k$-colorings,
thus partially answering also Problem~\ref{prob:general} (iii) (it remains open whether linearly many points per hyperedge/parallelogram suffice).

\begin{corollary}\label{cor:kcol}
	Let $Q$ be a given (open or closed) parallelogram and let $\mathcal{S}$ be a set of points in the plane.
	Then for every integer $k \geq 1$ it is possible to color $\mathcal{S}$ with $k$ colors, such that any homothet of $Q$ that contains at least $m_k=\Omega(k^{8.75})$ points from $\mathcal{S}$ contains points of all $k$ colors.
\end{corollary}

Our proof of Theorem~\ref{thm:parallelogram} also works for homothets of a triangle, i.e.,
we give a new proof for the primal case of Theorem~\ref{tri2col} (with a larger constant though):

\begin{theorem}[\cite{KP11}]
	\label{thm:triangle}
	There is an absolute constant $m_t$ such that the following holds.
	Given an (open or closed) triangle $T$ and a finite set of points in the plane $\mathcal{S} \subset \R^2$, 
	the points of $\mathcal{S}$ can be $2$-colored in polynomial time, such that any homothet of $T$ that contains
	at least $m_t$ points contains points of both colors.
\end{theorem}

This paper is organized as follows. In Section~\ref{sec:preliminaries} we introduce definitions, notations, tools and some useful lemmas.
In Section~\ref{sec:algorithm} we describe a general $2$-coloring algorithm and then apply it for parallelograms and for triangles.
Concluding remarks and open problems appear in Section~\ref{sec:discussion}.
A preliminary version of this paper was presented at the 32nd International Symposium on Computational Geometry (SoCG 2016). In the current version some of the proofs are simplified,
the constant in Theorem~\ref{thm:parallelogram} is better, and the limitations of our techniques are discussed in more details.

\section{Preliminaries}
\label{sec:preliminaries}

Unless stated otherwise, we restrict ourselves to the two-dimensional Euclidean space $\mathbb{R}^2$. 
For a point $p \in \mathbb{R}^2$ let $(p)_x$ and $(p)_y$ denote the $x$- and $y$-coordinate of $p$, respectively.
We denote by $\partial S$ the boundary of a subset $S \subseteq \mathbb{R}^2$ and by $Cl(S)$ the closure of $S$.
A \emph{homothet} of $S$ is a translated and scaled copy of $S$.
That is, a set $S'=\alpha S+p$ for some number $\alpha > 0$ and a point $p \in \mathbb{R}^2$. We will use the following folklore lemma:

\begin{lemma}
\label{lem:pseudo-disks}
Let $C$ be a convex and compact set and let $C_1$ and $C_2$ be homothets of $C$.
Then if $\partial C_1$ and $\partial C_2$ intersect finitely many times, then they intersect in at most two points.
\end{lemma}

For a proof of this lemma, see, e.g., \cite[Corollary 2.1.2.2]{Ma2000}.

\subsection{Generalized Delaunay triangulations}
\label{subsec:delaunay}

For proving Theorems~\ref{thm:parallelogram} and~\ref{thm:triangle} we will use the notion of generalized Delaunay triangulations, which are the dual of generalized Voronoi diagrams.
In the generalized Delaunay triangulation of a point set $\mathcal{S}$ with respect to some compact convex set $C$, two points of $\mathcal{S}$
are connected by a straight-line segment if there is a homothet of $C$ that contains these two points and does not contain any other point of $\mathcal{S}$.
The generalized Delaunay triangulation of $\mathcal{S}$ with respect to $C$ is denoted by $\DT(C,\mathcal{S})$.
Without causing confusion we also regard such a Delaunay triangulation as an abstract graph with $S$ as its vertex set and the above defined segments correspond to its edges.
We say that $\mathcal{S}$ is \emph{in general position} with respect to (homothets of) $C$,
if there is no homothet of $C$ whose boundary contains four points from $\mathcal{S}$.
If $\mathcal{S}$ is in general position with respect to a convex polygon $P$ and no two points of $\mathcal{S}$
define a line that is parallel to a line through two vertices of $P$, then we say that $\mathcal{S}$
is in \emph{very general position} with respect to $P$.
The following properties of generalized Delaunay triangulations will be useful.

\begin{lemma}[\cite{BCCS10,Kvoronoi,Sar10}]
\label{dfacts}
Let $C$ be a compact convex set and let $\mathcal{S}$ be a set of points in general position with respect to $C$.
Then $\DT(C,\mathcal{S})$ is a well-defined connected plane graph
whose inner faces are triangles.
\end{lemma}

It would be convenient to consider generalized Delaunay triangulations in which the boundary of the outer face is a convex polygon.
In such a case we say that $\DT(C,\mathcal{S})$ is \emph{nice}.


\begin{lemma}\label{lem:minus-z}
Let $P$ be a closed convex polygon and let $\mathcal{S}$ be a set of points in the plane that
is in very general position with respect to $P$.
Suppose that $P'$ is a homothet of $P$ and $Z \subseteq \mathcal{S} \cap \partial P'$.
Then there is a homothet of $P$, denote it by $P''$, such that $P'' \cap \mathcal{S} = (P' \cap \mathcal{S}) \setminus Z$.
\end{lemma}

\begin{proof}
	Since $\mathcal{S}$ is in general position, $| \partial P' \cap \mathcal{S} | \leq 3$.
	
	If $|Z|=3$, then there are no other points on $\partial P'$, thus shrinking $P'$ from an inner point gives us the required $P''$.
	
	If $|Z|=2$, then if there is no other point on $\partial P'$, then we can again shrink from an inner point slightly to get the required $P''$. Otherwise, there is exactly one point $q$ on $\partial P'$ besides the two points of $Z$. Now slightly shrink $P'$ from $q$. As the points are in very general position, the resulting homothet will contain exactly the points $(P' \cap \mathcal{S}) \setminus Z$.
	
	\vspace{2mm}
	
	Finally, suppose that $|Z|=1$ and let $Z=\{z\}$. We consider three cases.
	
	\vspace{1mm}
	
	\noindent{\bf Case 1:}  $ \partial P' \cap \mathcal{S}  = \{z\}$.
	In this case we slightly shrink $P'$ with respect to some point in its interior and obtain the desired homothet $P''$.
	
	\noindent{\bf Case 2:}  $| \partial P' \cap \mathcal{S} | = 2$.
	Let $\partial P' \cap \mathcal{S} = \{x,z\}$.
	Since $\mathcal{S}$ is in very general position, $x$ and $z$ are on different sides of $\partial P'$.
	Therefore if we slightly shrink $P'$ with respect to $x$, the resulting homothet of $P$ contains $x$,
	does not contain $z$ and contains all other points in $P' \cap \mathcal{S}$.
	
	\noindent{\bf Case 3:}  $| \partial P' \cap \mathcal{S} | = 3$.
	Let $\partial P' \cap \mathcal{S} = \{x,y,z\}$.
	In this case first we enlarge $P'$ from $z$ to get a homothet $P^{+}$. Since the points in $\mathcal{S}$ are in very general position with respect to $P$, this can be done so that $\partial P^{+} \cap \mathcal{S}=\{z\}$ and $P' \cap \mathcal{S}= P^{+} \cap \mathcal{S}$. Then
	by slightly shrinking $P^{+}$ with respect to an interior point we get the desired homothet $P''$.
\end{proof}

For a homothet $C'$ of a compact convex set $C$ we denote by $\DT(C,\mathcal{S})[C']$ the subgraph of $\DT(C,\mathcal{S})$
that is induced by the points of $\mathcal{S} \cap C'$.
Note that it is not the same as $\DT(C,\mathcal{S} \cap C')$,
however the following is true. 

\begin{lemma}\label{lem:connected}
Let $P$ be a closed convex polygon, let $\mathcal{S}$ be a set of points in very general position with respect to $P$,
and let $P'$ be a homothet of $P$.
Then $\DT(P,\mathcal{S})[P']$ is a connected graph that is contained in $P'$.
\end{lemma}

\begin{proof}
	Let $\DT:=\DT(P,\mathcal{S})$.
	Obviously all the points in $\DT[P']$ are in $P'$ by definition.
	An edge in $\DT[P']$ must also be in $P'$, since $P'$ is convex.
	We prove that $\DT[P']$ is connected by induction on $|\mathcal{S} \cap P'|$.
	This is true by definition if there are at most two points from $\mathcal{S}$ in $P'$.
	Suppose that the claim holds whenever a homothet of $P$ contains $k-1$ points from $\mathcal{S}$
	and let $P'$ be a homothet of $P$ that contains exactly $k$ points from $\mathcal{S}$ with $k \ge 3$.
	We may assume without loss of generality that $P'$ contains two points, $x$ and $z$, on its boundary,
	for otherwise we can continuously shrink $P'$ until such points exist (first from an interior point, then from the point that first appears on the boundary).
	Now apply Lemma~\ref{lem:minus-z} twice: once with $Z=\{x\}$ and once with $Z=\{z\}$. 
	Denote the homothets of $P$ that we get by $P_x$ and $P_y$, respectively.
	By the induction hypothesis $\DT[P_x]$ and $DT[P_z]$ are both connected, their intersection contains at least one point (as $k \ge 3$), and their union is contained in $\DT[P']$. Thus, $\DT[P']$ is also connected, as required.	
\end{proof}

\begin{corollary}
\label{cor:split}
Let $P$ be a closed convex polygon and let $\mathcal{S}$ be a set of points in very general position with respect to $P$.
Suppose that $P'$ is a homothet of $P$ and $e$ is an edge of $\DT(P,\mathcal{S})$ that crosses $\partial P'$ twice and thus splits $P'$ into two parts.
Then one of these parts does not contain a point from $\mathcal{S}$.
\end{corollary}

A \emph{rotation} of a vertex $v$ in a plane graph $G$ is the clockwise order of its neighbors.
For three straight-line edges $vx,vy,vz$ we say that $vy$ is \emph{between} $vx$ and $vz$ if $x,y,z$ appear in this order in the rotation of $v$ and $\angle xvz < \pi$
($\angle xvz$ is the angle by which one has to rotate the vector $\vec{vx}$ around $v$ clockwise until its direction coincides with that of $\vec{vz}$) or if $z,y,x$ appear in this order in the rotation of $v$ and  $\angle zvx < \pi$.
The following will be useful later on.

\begin{proposition}
\label{prop:consecutive}
Let $C$ be a compact convex set and let $\mathcal{S}$ be a set of points in very general position with respect to $C$ and such that $\DT:=\DT(C,\mathcal{S})$ is nice.
Let $C'$ be a homothet of $C$ and let $v$ be a vertex in $\DT[C']$.
Suppose that $x$ and $z$ are two vertices that are neighbors of $v$ in $\DT[C']$ and $\angle xvz < \pi$ and $xz \notin \DT$.
Then there exists an edge $vy \in \DT$ between $vx$ and $vz$.
Moreover, if $z$ immediately follows $x$ in the rotation of $v$ in $\DT[C']$ then $y \notin C'$.
\end{proposition}

\begin{proof}
If $x$ and $z$ are not consecutive in the rotation of $v$ in $\DT$ then by definition there exists a vertex $y$ between them in the rotation of $v$.

Thus we are done unless $x$ and $z$ are consecutive in the rotation of $v$ in $\DT$.

Suppose that such an $y$ does not exist, that is, $x$ and $z$ are also consecutive in the rotation of $v$ in $\DT$.
Then the face that is incident to $vx$ and $vz$ and is to the right of $\vec{vx}$ and to the left of $\vec{vz}$
cannot be the outer face since $\angle xvz < \pi$ and $\DT$ is nice.
However, since this face is an inner face, then by Lemma~\ref{dfacts} it must be a triangle and so $xz \in \DT$.

Thus we can conclude that such an $y$ exists and from the definition of the rotation order it follows that $y \notin C'$.
\end{proof}

 \begin{lemma}
 \label{lem:degree}
 For every closed convex polygon $P$ there is a constant $\Delta:=\Delta(P)$ such that the following holds.
 Let $\mathcal{S}$ be a set of points in very general position with respect to $P$
 and such that $\DT:=\DT(P,\mathcal{S})$ is nice.
 If $P'$ is a homothet of $P$ such that $\DT[P']$ is a tree,
 then for every $v \in \mathcal{S} \cap P'$ we have $\deg_{\DT[P']}(v) \leq \Delta$.
 \end{lemma}

 \begin{proof}
 Let $n$ be the number of vertices of $P$ and let $v_0,v_1,\ldots,v_{n-1}$ be the vertices of $P'$ listed in their clockwise order around $P'$.
 Let $v \in \mathcal{S} \cap P'$ be a point and let $N_i:=\{u \in \mathcal{S} \cap P' \mid u \in \triangle vv_iv_{i+1}, vu \in \DT[P'] \}$
 be the neighbors of $v$ in $\DT[P']$ that are also in the triangle $vv_iv_{i+1}$, for every $i=0,\ldots,n-1$ (addition is modulo $n$).
 Let $\alpha:=\alpha(P)$ be the smallest angle formed by three vertices of $P$ (hence, also of $P'$).

 \begin{obs}
 \label{obs:inner}
 For every point $p'$ in the interior of $P'$ and every $0 \leq i<j \leq n-1$ we have $\angle v_ip'v_j \geq \alpha$.
 \end{obs}

 \begin{proposition}
 \label{prop:alpha}
 For every $i=0,\ldots,n-1$ and every $u,u' \in N_i$ we have $\angle uvu' \geq \alpha$.
 \end{proposition}

\begin{proof}
	It is enough to consider the case when $u$ and $u'$ follow each other immediately in the rotation of $v$. First note that $uu'$ is not an edge in $\DT[P']$, since otherwise there would be a triangle in $\DT[P']$. It follows from Proposition~\ref{prop:consecutive} that there is a point $z \notin P'$ such
	that $z$ is a neighbor of $v$ in $\DT$ and is between $u$ and $u'$ in the rotation of $v$.
	Thus, there is a homothet of $P$, denote it by $P_z$, such that $P_z \cap \mathcal{S} = \{v,z\}$.
	Let $v'_iv'_{i+1}$ be the side of $P_z$ that corresponds and is parallel to the side $v_iv_{i+1}$ of $P'$ (see Figure~\ref{fig:alpha}).
	Note that $v'_iv'_{i+1}$ is outside of $P'$, since $z \notin P'$.
	Furthermore, since $u,u' \notin P_z$, the side $v'_iv'_{i+1}$ lies inside the wedge whose apex is $v$ and whose boundary consists
	of the two rays that emanate from $v$ and go through $u$ and $u'$, respectively.
	Therefore, $\angle v'_ivv'_{i+1} < \angle uvu'$.
	It follows from Observation~\ref{obs:inner} that $\angle v'_ivv'_{i+1} \geq \alpha$, thus we have $\angle uvu' \geq \alpha$.
	\begin{figure}
		\centering
		\includegraphics[width=6cm]{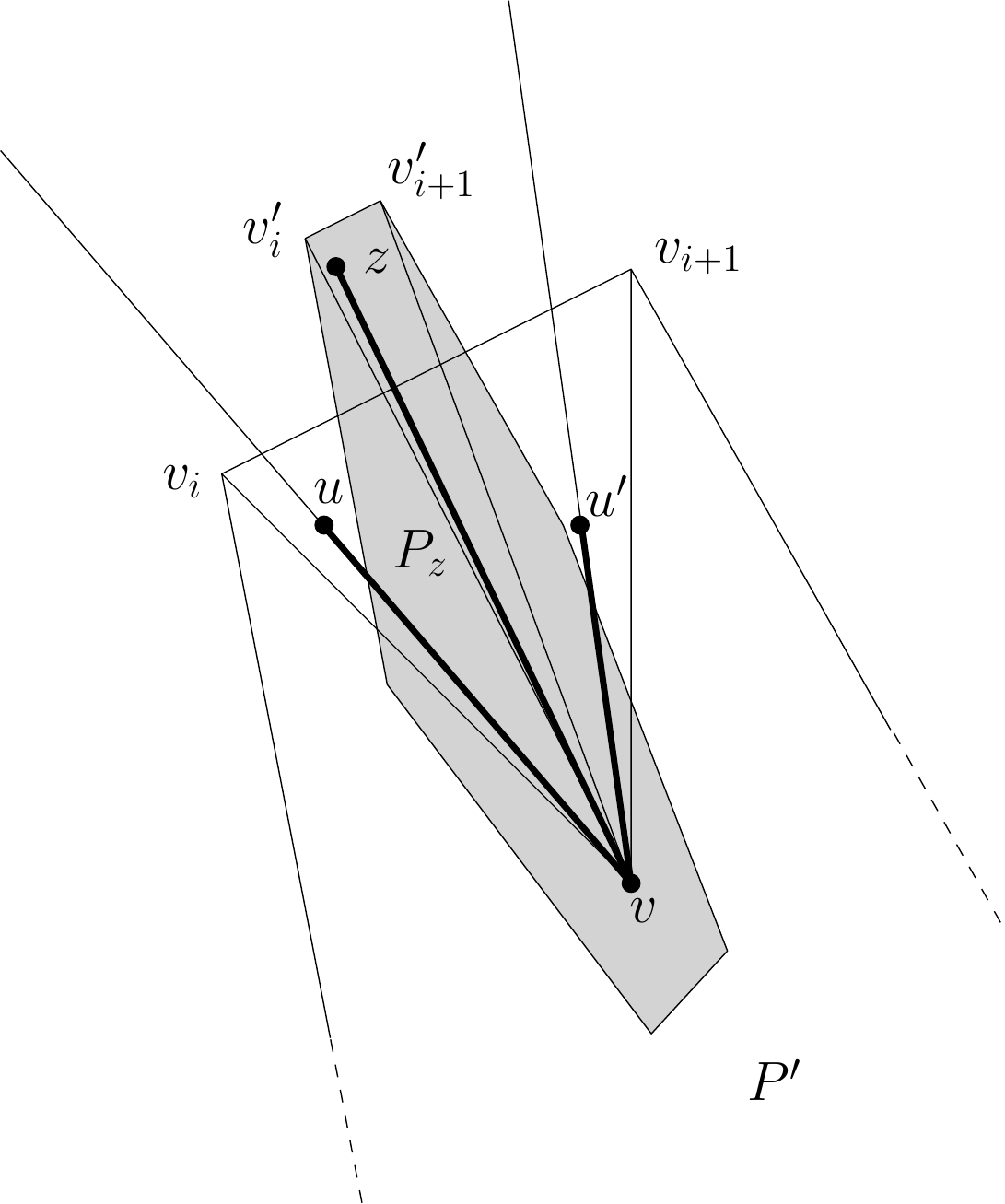}
		\caption{An illustration for the proof of Proposition~\ref{prop:alpha}.}
		\label{fig:alpha}
	\end{figure}
\end{proof}


 To complete the proof of Lemma~\ref{lem:degree} consider the neighbors of $v$ in $\DT[P']$ in their clockwise order around $v$,
 and for every set $N_i$ remove the extreme neighbor in this order.
 It follows from Proposition~\ref{prop:alpha} that the angle between any two remaining neighbors of $u$ is at least $\alpha$.
 Therefore, $\deg_{\DT[P']}(v) \leq n + \frac{2\pi}{\alpha}$.
\end{proof}

It follows that if $P$ is an equilateral triangle, then Lemma~\ref{lem:degree}
applies with $\Delta(P) \leq 3 + \frac{2\pi}{\pi/3}=9$.
By affine transformations we have:

\begin{corollary}
	\label{cor:triangleangles}
	Suppose that $T$ is a triangle and $\mathcal{S}$  is a set of points in very general position with respect to $T$ and such that $\DT:=\DT(T,\mathcal{S})$ is nice. If $T'$ is a homothet of $T$ such that $\DT[T']$ is a tree, then for every point $v \in T' \cap \mathcal{S}$ we have $\deg_{\DT[T']}(v) \leq 9$.
\end{corollary}

 \begin{corollary}
 \label{cor:long-path}
 Let $P$ be convex polygon and let $\Delta:=\Delta(P)$ be the constant from Lemma~\ref{lem:degree}.
 Suppose that $\mathcal{S}$ is a set of points in very general position with respect to $P$
 and such that $\DT:=\DT(P,\mathcal{S})$ is nice.
 If $P'$ is a homothet of $P$ such that $\DT[P']$ is a tree,
 then $\DT[P']$ contains a simple path of length at least $2\lfloor\log_{\Delta}|\mathcal{S} \cap P'|\rfloor$.
 \end{corollary}

\subsection{Self-coverability of convex polygons and polychromatic $k$-coloring} \label{subsec:self-coverability}

Keszegh and P\'alv\"olgyi introduced in~\cite{selfcover} the notion of \emph{self-coverability}
and its connection to polychromatic $k$-coloring.
In this section we list the definition and results from their work that we use.

\begin{definition}[\cite{selfcover}]\label{selfcov}
\label{def:self-coverable}
A collection of closed sets $\cal F$ in a topological space is {\em self-coverable} if there exists a {\em self-coverability function} $f$ such that
for any set $F \in \cal F$ and for any finite point set $\mathcal{S} \subset F$, with $|\mathcal{S}|=l$
there exists a subcollection ${\cal F}' \subset {\cal F}$, $|{\cal F'}|\le f(l)$
such that $\bigcup_{F' \in \cal F'} F' = F$ and no point of $\mathcal{S}$ is in the interior of some $F' \in {\cal F}'$.
\end{definition}

\begin{theorem}[\cite{selfcover}]
\label{thm:convexsc}
For every convex polygon $P$ there is a constant $c_f:=c_f(P)$ such that
the family of all homothets of $P$ is self-coverable with $f(l) \leq c_f l$.
\end{theorem}

\begin{theorem}[\cite{selfcover}]
\label{thm:squaressc}
The family of all homothets of a square is self-coverable with $f(l):=2l+2$ and this is sharp.
\end{theorem}

\begin{theorem}[\cite{selfcover}]
\label{thm:trianglesssc}
The family of all homothets of a given triangle is self-coverable with $f(l):=2l+1$ and this is sharp.
\end{theorem}

\begin{theorem}[{\cite[Theorem~2]{selfcover}}]\label{thm:connection}
If $\cal F$ is self-coverable with a monotone 
self-coverability function $f(l)>l$ and any finite set of points can be colored with two colors such that any member of $\cal F$ with at least $m$ points contains both colors, then any finite set of points can be colored with $k$ colors such that any member of $\cal F$ with at least $m_k:= m (f(m-1))^{\lceil \log k\rceil-1}\le k^d$ points contains all $k$ colors (where $d$ is a constant that depends only on $\cal F$).\footnote{Unless stated otherwise, logarithms in this paper are base $2$.}
\end{theorem}

Theorem~\ref{thm:parallelogram} (which we have yet to prove) and Theorems~\ref{thm:squaressc} and~\ref{thm:connection} immediately imply Corollary \ref{cor:kcol}. Indeed, the required assumptions of Theorem \ref{thm:connection} hold for squares with $m \leq 215$ by Theorem ~\ref{thm:parallelogram} and $f(l)=2l+2$ by Theorem~\ref{thm:squaressc}. We get $m_k=f(m-1)^{\lceil \log k\rceil-1}\le (2m)^{\log k}= k^{\log 430}=O(k^{8.75})$ for squares, and also for parallelograms by affine transformations.

\section{A $2$-coloring algorithm}
\label{sec:algorithm}

In this section we prove Theorems~\ref{thm:parallelogram} and~\ref{thm:triangle}.
In fact, we prove a more general result, for which we need the following definitions.

\begin{definition}[Good paths and good homothets]
Let $P$ be an (open or closed) convex polygon, let $\mathcal{S}$ be a finite set of points, 
let $\DT:=\DT(P,\mathcal{S})$, and let $P'$ be a homothet of $P$.
\begin{packed_item}
\item Let $\twopath{x}{y}{z}$ be a $2$-path in $\DT$ (i.e., a simple path of length two).
If $Cl(P')$ does not contain $x$ and $z$ and $y$ is in the interior of $P'$, then we say that it \emph{separates} the $2$-path $\twopath{x}{y}{z}$.
\item A $2$-path $\twopath{x}{y}{z}$ is \emph{good}, if there is no homothet of $P$ that separates it such that the edges $yx$ and $yz$ cross
the same side of this homothet of $P$ (see Figure~\ref{fig:2-paths} for an example).
\begin{figure}
    \centering
    \includegraphics[width=3cm]{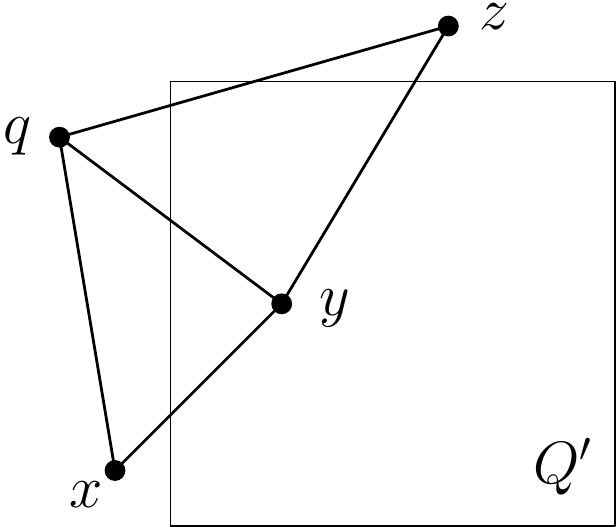}
	\caption{Considering homothets of an axis-parallel square, $\twopath{x}{y}{z}$ is a good $2$-path whereas $\twopath{x}{y}{q}$ is not
	since the square $Q'$ separates it and both $xy$ and $qy$ cross the left side of $Q'$.}
	\label{fig:2-paths}
\end{figure}
\item A $3$-path $\threepath{x}{y}{z}{w}$ in $\DT$ is \emph{good} if both $\twopath{x}{y}{z}$ and $\twopath{y}{z}{w}$ are good $2$-paths.
\item $P'$ is \emph{good} if it contains a good $3$-path or $\DT[P']$ contains a cycle.
\end{packed_item}
\end{definition}

Observe that whether a $2$-path $\twopath{x}{y}{z}$ is good depends only on the direction of the vectors $\vec{yx}$ and $\vec{yz}$.

\begin{definition}[Universally good polygons]
We say that an (open or closed) convex polygon $P$ is \emph{universally good} with a constant $c_g:=c_g(P)$
if for any finite set of points $\mathcal{S}$ such that $\mathcal{S}$ is in very general position with respect to $P$ and $\DT(P,\mathcal{S})$ is nice,
every homothet of $P$ that contains at least $c_g$ points from $\mathcal{S}$ is good.
\end{definition}

\begin{theorem}
\label{thm:general}
Let $P$ be an (open or closed) convex polygon with $n$ vertices such that $P$ is a universally good polygon with a constant $c_g:=c_g(P)$,
and let $f(l) \leq c_f l$ be a self-coverability function of the family of homothets of $Cl(P)$ (where $c_f:=c_f(P)$ is a constant).
Then there is a constant $m:=m(P) \leq (c_g-1)f(n)+n+1\leq (c_g-1)c_f n+n+1$ such that it is possible to $2$-color in polynomial time
the points of any given finite set of points $\mathcal{S}$ 
such that every homothet of $P$ that contains at least $m$ points from $\mathcal{S}$ contains points of both colors.
\end{theorem}

We note that for the existence of $m$ in Theorem \ref{thm:general} we need only that $f(l)$ exists. However, using that it is upper bounded by $c_f l$ \cite{selfcover} we can get a more explicit bound on $m$.

Theorems~\ref{thm:parallelogram} and~\ref{thm:triangle} immediately follow from Theorems~\ref{thm:squaressc},
\ref{thm:trianglesssc}, \ref{thm:general}, and the following two lemmas.

\begin{lemma}
\label{lem:triangle}
Every triangle is a universally good polygon with a constant $c_g \leq 7382$.
\end{lemma}

\begin{lemma}
\label{lem:parallelogram}
Every parallelogram is a universally good polygon with a constant $c_g \leq 22$.
\end{lemma}

In particular, we get the value in Theorem \ref{thm:parallelogram} by taking $m=(c_g-1)f(n)+n+1=215$ with $c_g=22, f(l)=2l+2$ and $n=4$ for squares.

In light of Theorem~\ref{thm:general}, it is enough to prove that a convex polygon is universally
good to conclude that Problem~\ref{prob:general}(i) has a positive solution with respect to homothets
of that polygon.
However, as it turns out, parallelograms and triangles are the only universally good polygons.

\begin{theorem}\label{thm:construction}
	Let $P$ be a convex polygon which is neither a triangle nor a parallelogram.
	Then $P$ is not universally good.
\end{theorem}

We proceed with the proof of Theorem~\ref{thm:general}, then prove that triangles and parallelograms
are universally good, and conclude this section with a proof that no other universally good polygons exist.

\subsection{Proof of Theorem~\ref{thm:general}}
\label{subsec:main-proof}

Let $P$ be an (open or closed) convex polygon with $n$ vertices and let $\bar{P}:=Cl(P)$ be the closure of $P$ (thus, $\bar{P}$ is a closed polygon and $\bar{P}=P$ if $P$ is closed). Let us assume also that $P$ is a universally good polygon with a constant $c_g:=c_g(P)$,
and let $f(l) \leq c_f l$ be a self-coverability function of the family of homothets of $\bar{P}$.
Set $m:= (c_g-1)f(n)+n+1\le (c_g-1)c_f n+n+1$.
We first argue that it is enough to prove Theorem~\ref{thm:general} when $P$ is a closed polygon.
Indeed, suppose that $P$ is open and
let $\cal P$ be the family of homothets of $P$.
By slightly shrinking every homothet of $P$ in ${\cal P}$ with respect to an interior point,
we get a family ${\cal P}'$ of homothets that realizes $H^{\cal P}(\cal S)$ such that there is no point $p \in \mathcal{S}$ and homothet $P' \in  {\cal P}'$ with $p \in \partial P'$.

Note that by definition $\bar{P}$ is universally good with the same constant $c_g$ and is self-coverable with the same self-coverability function as $P$.
Let $\bar{\cal P'} := \{Cl(P') | P' \in {\cal P}'\}$.
Since there is no homothet of $P$ in ${\cal P}'$ that contains a point of $\mathcal{S}$ on its boundary,
every hyperedge of $H^{\cal P}(\cal S)$ appears also in $H^{\bar{\cal P'}}(\cal S)$.
Thus, it is enough to show that $\bar{P}$ satisfies Theorem~\ref{thm:general}.

\medskip

Suppose therefore that $P$ is a closed convex polygon.
Let ${\cal P}$ be the family of homothets of $P$ and let ${\cal P}_0 \subseteq \mathcal{P}$ be a smallest size subfamily that realizes $H^{\cal P}(\cal S)$.
For convenience we pick ${\cal P}_0$ such that no $P' \in {\cal P}_0$ contains a point from $\mathcal{S}$ on its boundary (this can be achieved by slightly inflating homothets if necessary).

We may also assume that $\mathcal{S}$ is in very general position with respect to $P$.
Indeed, otherwise note that a small perturbation of the points will achieve that while ${\cal P}_0$ will still realize the same hypergraph $H^{\cal P}(\cal S)$. It will also be convenient to assume that the boundaries of every two polygons in ${\cal P}_0$ do not overlap,
and no edge in $\DT:=\DT(P,\mathcal{S})$ crosses the boundary of a polygon in $\mathcal{P}_0$ at one of its vertices.
It follows from Lemma~\ref{lem:pseudo-disks} that $\mathcal{P}_0$ is a family of \emph{pseudo-disks}.\footnote{In a family of pseudo-disks the boundaries of every two regions cross at most twice.}
This implies that $|{\cal P}_0| = O(|\mathcal{S}|^3)$ by a result of Buzaglo et al.~\cite{BPR13} who proved the following:
Suppose that $(\mathcal{V},\mathcal{E})$ is a hypergraph where $\mathcal{V}$ is a set of points in the plane
and for every hyperedge $e \in \mathcal{E}$ there is a region bounded by a simple closed curve that contains the points of $e$ and no other points from $\mathcal{V}$.
If all the regions that correspond to $\mathcal{E}$ define a family of pseudo-disks, then $|\mathcal{E}|=O(|\mathcal{V}|^3)$.

We can also assume that $\DT$ is nice, that is, the boundary of its outer face is a convex polygon:
Set $-P := \{ (-x,-y) \mid (x,y) \in P \} $ and let $-P'$ be a homothet of $-P$ that contains in its interior all
the polygons in $\mathcal{P}_0$.
By adding the vertices of $-P'$ to $\mathcal{S}$ (and perturbing again if needed) we obtain a set of points $\mathcal{S'}$
such that $-P'$ is the boundary of the outer face in its generalized Delaunay triangulation with respect to $P$.
Moreover, the hypergraph we get by intersecting homothets of $P$ with $\mathcal{S'}$ is a hypergraph that contains all the hyperedges that we get by intersecting polygons in $\mathcal{P}_0$ with $\mathcal{S}$. The latter hyperedges are exactly the hyperedges we get by intersecting homothets of $P$ with $\mathcal{S}$, since ${\mathcal P}_0$ realizes $H^{\cal P}(\cal S)$. Therefore a valid $2$-coloring of the new set of points induces a valid $2$-coloring of the original set of points.

Recall that $\DT$ is a plane graph, and therefore, by the Four Color Theorem, we can color the points in $\mathcal{S}$
with four colors, say $1,2,3,4$, such that there are no adjacent vertices in $\DT$ with the same color.
In order to obtain two color classes, we recolor all the vertices of colors $1$ or $2$ with the color \emph{light red}
and all the vertices of colors $3$ or $4$ with the color \emph{light blue}.

Call a homothet $P' \in \mathcal{P}_0$ \emph{heavy monochromatic} if it contains exactly $c_g$ points from $\mathcal{S}$ and all of them are of the same light color.
If all of these points are colored light blue (resp., red), then we call $P'$ a heavy light blue (resp., red) homothet.
Obviously, if there are no heavy monochromatic homothets, then we are done since $m > c_g$ and it follows from Lemma \ref{lem:minus-z} that a monochromatic homothet with $m>c_g$ points from $\mathcal{S}$ can be shrinked to a monochromatic homothet with exactly $c_g$ points from $\mathcal{S}$.

Suppose that $P'$ is a heavy monochromatic homothet of $P$.
Observe that $\DT[P']$ is a tree, for otherwise it would contain a cycle which in turn would contain a triangle by Lemma~\ref{dfacts}.
That triangle must be $3$-colored in the initial $4$-coloring, so not all of its points can be light red or light blue,
contradicting the monochromaticity of the points in $P'$.

Since $P$ is universally good, $P'$ contains $c_g$ points and $\DT[P']$ is a tree, it follows that
$P'$ contains a good $3$-path $\threepath{x}{y}{z}{w}$.
We associate this $3$-path with $P'$.
Suppose that $P'$ is a heavy light red homothet of $P$.
Then one of $y$ and $z$ was originally colored $1$ and the other was originally colored $2$.
Recolor the one whose original color was $1$ with the color \emph{dark blue}.
Similarly, if $P'$ is a heavy light blue homothet of $P$,
then one of $y$ and $z$ was originally colored $3$ and the other was originally colored $4$.
In this case we recolor the one whose original color was $3$ with the color \emph{dark red}.
Repeat this for every heavy monochromatic homothet, and, finally, in order to obtain a $2$-coloring, merge the color classes light red and dark red into one color class, red;
and merge the color classes light blue and dark blue into one color class, blue.

\begin{lemma}\label{lem:2-coloring}
There is no homothet $P' \in \mathcal{P}_0$ that contains $m$ points from $\mathcal{S}$ all of which are of the same color.
\end{lemma}

\begin{proof}
Suppose for contradiction that $P'$ is a homothet of $P$ that contains $m$ points from $\mathcal{S}$
all of which of the same color.
We may assume without loss of generality that all the points in $P'$ are colored red,
therefore, before the final recoloring each point in $P'$ was either light red or dark red.
We consider two cases based on the number of dark red points in $P'$. Recall that $n$ is the number of vertices of $P$.

\paragraph*{Case 1: There are at most $n$ dark red points in $P'$.}
By Definition~\ref{selfcov} there is a set $\mathcal{P}'$ of at most $f(n)$ homothets of $P$ whose union is $P'$
such that no dark red point in $P'$ is in the interior of one of these homothets.
Using Lemma \ref{lem:minus-z} we can change these homothets slightly such that none of them contains a dark red point yet all light red points are still covered by these homothets.
Thus the at least $m-n=(c_g-1)f(n)+1$ light red points are covered by these at most $f(n)$ homothets.
By the pigeonhole principle one of these homothets, denote it by $P''$, contains at least $\lceil\frac{(c_g-1)f(n)+1}{f(n)}\rceil=c_g$ light red points and no other points.
However, in this case it follows from Lemma~\ref{lem:minus-z} that there is a heavy light red
homothet in $\mathcal{P}_0$ that contains exactly $c_g$ points from $\mathcal{S} \cap P''$.
Therefore, the coloring algorithm should have found within this heavy light red homothet a good $3$-path and recolored one of its vertices with dark blue and then blue.
This contradicts the assumption that all the points in $P'$ are red.

\paragraph*{Case 2: There are more than $n$ dark red points in $P'$.}
Let $y$ be one of these dark red points.
Then there is a good $3$-path $\threepath{x}{y}{z}{w}$ within a heavy light blue homothet $P_y \in \mathcal{P}_0$ with whom this $3$-path is associated.
Furthermore, the original color of $y$ is $3$ and therefore the original color of $x$ and $z$ is $4$,
and thus their final color is blue.
It follows that $P'$ separates $\twopath{x}{y}{z}$, moreover, since $\twopath{x}{y}{z}$ is a good $2$-path, the edges $yx$ and $yz$ cross different sides of $P'$.
Let $s_x$ be the side of $P'$ that is crossed by $yx$,
and let $q_x$ be the crossing point of $yx$ and $s_x$.
Similarly, Let $s_z$ be the side of $P'$ that is crossed by $yz$,
and let $q_z$ be the crossing point of $yz$ and $s_z$.
See Figure~\ref{fig:in-R_y}.
\begin{figure}
    \centering
    \includegraphics[width=8cm]{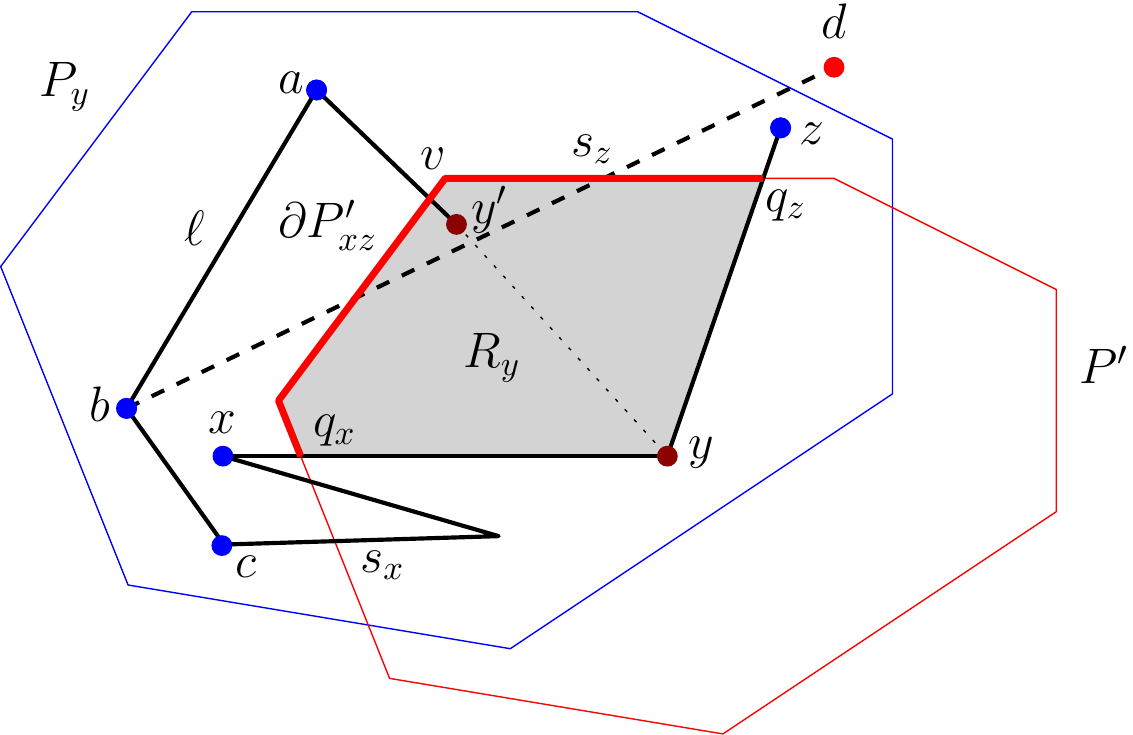}
	\caption{An illustration for the proof of Lemma~\ref{lem:2-coloring}.}
	\label{fig:in-R_y}
\end{figure}
Note that $\partial P'$ and $\partial P_y$ cross each other exactly twice.
Indeed, this follows from Lemma~\ref{lem:pseudo-disks} and the fact that
there are points from $\mathcal{S}$ in each of $P' \cap P_y$ (e.g., $y$),
$P_y \setminus P'$ (e.g., $x$ and $z$) and $P' \setminus P_y$ (since $|P' \cap \mathcal{S}| \geq m > c_g = |P_y \cap \mathcal{S}|$).
The points $q_x$ and $q_z$ partition $\partial P'$ into two parts $\partial P'_1$ and $\partial P'_2$.
Note that since $q_x,q_z \in P' \cap P_y$, the two crossing points between $\partial P'$ and $\partial P_y$
must lie either in $\partial P'_1$ or in $\partial P'_2$.
Assume without loss of generality that both of them lie in $\partial P'_1$.
Thus $\partial P'_2 \subset P_y$.
Let $v$ be a vertex of $P'$ in $\partial P'_2$ (note that since $s_x \neq s_z$ each of $\partial P'_1$ and $\partial P'_2$ contains a vertex of $P'$).
We associate the vertex $v$ with the dark red point $y$.
We also define $R_y$ to be the region whose boundary consists of the segments $yq_{x}$, $yq_{z}$,
and the part of $\partial P'_2$ whose endpoints are $q_x$ and $q_z$ (call this part $\partial P'_{xz}$).
Observe that $R_y \subseteq P' \cap P_y$.

\begin{proposition}\label{prop:y-y'}
There is no other point but $y$ in ${\cal S}\cap R_y$.
\end{proposition}

\begin{proof}
Suppose that the claim is false and let $y' \in \mathcal{S} \cap R_y$ be another point in $R_{y}$. As $y'$ is in $P'$, it must be red after the final coloring. Also, as it is also in $P_y$, it must be a dark red point (which was light blue before having been recolored to dark red and finally to red). Thus, $y'$ is a dark red point in $R_{y}$.

Since $x$ and $y'$ both lie in the heavy light blue homothet $P_y$, they are connected by a path in $\DT[P_y]$
that alternates between points of colors $3$ and $4$ (considering the initial $4$-coloring).
We may assume without loss of generality that $y'$ is the first point in $R_y$ along this path from $x$ to $y'$:
indeed, there are no points of color $4$ in $R_y$, and if there is a point of color $3$ before $y'$,
then we can name it $y'$.
Denote by $\ell$ the path (in $\DT$) from $y$ to $y'$ that consists of the edge $yx$ and the above-mentioned path from $x$ to $y'$.
Consider the polygon $\hat{P}$ whose boundary consists of $\ell$ and a straight-line segment $yy'$ ($\hat{P}$ is not a homothet of $P$).
Since $y'$ and $y$ are the only vertices of $\hat{P}$ in $R_y$, there is no edge of $\ell$ that crosses $yy'$.
Indeed, if there was such an edge, then it would split $P'$ into two parts such
that one contains $y$ and the other contains $y'$.
This would contradict Corollary~\ref{cor:split}.
Hence $\hat{P}$ is a simple polygon.

Since every simple polygon has at least three convex vertices,
$\hat{P}$ has a convex vertex different from $y$ and $y'$ (thus this vertex is not in $R_y$).
Denote this vertex by $b$ and let $a$ and $c$ be its neighbors along $\ell$ such that $\angle abc < \pi$. Since the initial colors of $a$ and $c$ are the same, we have $ac \notin \DT$ and so
it follows from Proposition~\ref{prop:consecutive} that there is a neighbor $d$ of $b$ in $\DT$ in between $a$ and $c$. Let us choose $d$ such that it is the neighbor of $b$ that is closest to $a$ in the rotation of $b$. Thus, it is connected to both $a$ and $b$. Since the initial colors of $a$ and $b$ are $3$ and $4$, the initial color of $d$ was $1$ or $2$.
Note that $\hat{P} \subseteq P_y$ since all of its edges are inside $P_y$.
Thus $d \notin \hat{P}$ and also $d \notin P_y$ since $P_y$ does not contain vertices of color $1$ or $2$.
Now consider the directed edge $bd$: it starts inside $\hat{P}$ (since $d$ is in between $a$ and $c$) and so it must cross $yy'$.
Before doing so $bd$ must cross $\partial R_y$ and so it crosses $\partial P'_{xz}$,
since it cannot cross $yq_z$ or $yq_x$.
After crossing $yy'$, the directed edge $bd$ must cross $\partial P'_{xz}$ again, since $d \notin R_y$.
But then $bd$ splits $P'$ into two parts such that one contains $y$ and the other contains $y'$,
which is impossible by Corollary~\ref{cor:split}.
\end{proof}

In a similar way to the one described above, we associate a vertex of $P'$ with every dark red point in $P'$.
Since there are more than $n$ dark red points in $P'$, there are two of them, denote them by $y$ and $y'$,
that are associated with the same vertex of $P'$, denote it by $v$.
Let $\twopath{x}{y}{z}$ be the good $2$-path that corresponds to $y$, 
let $yq_x$ and $yq_z$ be the edge-segments of $yx$ and $yz$,
and let $R_y$ be the region as defined above.
Similarly, let $\twopath{x'}{y'}{z'}$ be the good $2$-path that corresponds to $y'$,.
let $yq_{x'}$ and $yq_{z'}$ be the edge-segments of $y'x'$ and $y'z'$,
and let $R_{y'}$ be the region as defined above.

It follows from Proposition~\ref{prop:y-y'} that $y \notin R_{y'}$ and $y' \notin R_{y}$.
However, $\partial R_y$ and $\partial R_y'$ both contain $v$. This implies
that one of the segments $yq_{x}$ and $yq_{z}$ crosses one of the segments $y'q_{x'}$ and $y'q_{z'}$,
which is impossible since these are segments of edges of a plane graph.
Lemma~\ref{lem:2-coloring} is proved.
\end{proof}

To complete the proof of Theorem~\ref{thm:general}, we need to argue that the described algorithm runs in polynomial time.
Indeed, constructing the generalized Delaunay triangulation and then $4$-coloring it can be done in polynomial time.
Recall that there are at most $O(|\mathcal{S}|^3)$ combinatorially different homothets of $P$.
Among them, we need to consider those that contain exactly $c_g$ points, and for each such heavy monochromatic homothet $P'$
we need to find a good $3$-path in $\DT[P']$, for the final recoloring step.
This takes a constant time for every heavy monochromatic homothet, since $c_g$ is a constant.
Therefore, the overall running time is polynomial with respect to the size of $\mathcal{S}$.\hfill$\qed$

\subsection{Triangles are universally good}
\label{subsec:triangles}

In this section we prove Lemma~\ref{lem:triangle}.

	Let $T$ be a triangle, let $\mathcal{S}$ be a set of points in very general position with respect to $T$,
	and let $\DT:=\DT(T,\mathcal{S})$ be the generalized Delaunay triangulation of $\mathcal{S}$ with respect to $T$ such that $\DT$ is nice
	(i.e., the boundary of its outer face is a convex polygon).
	By applying an affine transformation, if needed, we may assume without loss of generality that $T$ is an equilateral triangle.
	Suppose that $T'$ is a homothet of $T$ that contains at least $7382$ points from $\mathcal{S}$ and that $\DT[T']$ is a tree.
	We will show that $T'$ contains a good $3$-path.
	
	By Corollary \ref{cor:triangleangles} for every point $v \in T' \cap \mathcal{S}$ we have $\deg_{\DT[T']}(v) \leq 9$.
	Since $\DT[T']$ is a tree with at least $7382 = 1+9+9^2+9^3+9^4+1$ vertices of maximum degree $9$, it contains a simple path of length $9$.
	Let $Z=v_1 \mhyphen v_2 \mhyphen \ldots \mhyphen v_{10}$ be such a path.
	We will prove that there is $2 \leq i \leq 8$ such that $\twopath{v_{i-1}}{v_{i}}{v_{i+1}}$ and $\twopath{v_{i}}{v_{i+1}}{v_{i+2}}$ are good $2$-paths,
	and therefore $T$ contains the good $3$-path $\threepath{v_{i-1}}{v_{i}}{v_{i+1}}{v_{i+2}}$.
	Call a $2$-path $\twopath{v_{i-1}}{v_{i}}{v_{i+1}}$ (for $2 \leq i \leq 9$) \emph{bad} if it is not good, that is,
	there is a homothet of $T$, $T_i$, such that $T_i$ contains $v_i$, does not contain $v_{i-1}$ and $v_{i+1}$,
	and the edges $v_iv_{i-1}$ and $v_iv_{i+1}$ cross the same side of $T_i$.
	
	Denote the sides of $T$ by $s_1,s_2,s_3$.
	For $j=1,2,3$, let $B_j$ be the set of bad $2$-paths $\twopath{v_{i-1}}{v_{i}}{v_{i+1}}$
	such that there is a homothet $T_i$ that contains $v_i$ and does not contain $v_{i-1}$ and $v_{i+1}$,
	and the edges $v_iv_{i-1}$ and $v_iv_{i+1}$ both cross the side of $T_i$ that is homothetic to $s_j$.
	Suppose for contradiction that $Z$ does not contain two consecutive good $2$-paths.
	Then, at least one of the sets $B_j$ contains two bad $2$-paths.
	Assume without loss of generality that $B_1$ contains two bad $2$-paths $\twopath{v_{i-1}}{v_{i}}{v_{i+1}}$ and $\twopath{v_{k-1}}{v_{k}}{v_{k+1}}$ such that $i<k$.
	We may further assume that $s_1$ is horizontal and that $T$ lies above it.
	
	There is a homothet of $T$ that separates $\twopath{v_{i-1}}{v_{i}}{v_{i+1}}$ such that  $v_{i-1}v_i$ and $v_{i}v_{i+1}$ both
	cross its side that is homothetic to $s_1$, therefore both $v_{i-1}$ and $v_{i+1}$ lie below $v_i$.
	Similarly, both $v_{k-1}$ and $v_{k+1}$ lie below $v_k$.
	Let $v_r$ be the lowest point among $v_{i},\ldots,v_{k}$.
	Since $v_{i+1}$ is lower than $v_i$ and $v_{k-1}$ is lower than $v_k$ it follows that $r \neq i,k$
	and so $v_r$ is lower than $v_{r-1}$ and $v_{r+1}$.
	Suppose without loss of generality that $v_{r+1}$ is to the right of the line through $v_r$ and $v_{r-1}$.
By applying Proposition~\ref{prop:consecutive} to two vertices that are between $v_{r-1}$ and $v_{r+1}$ and follow each other immediately in the rotation of $v_r$ in $\DT[T']$ (they can coincide with $v_{r-1}$ and/or $v_{r+1}$), we get that there is at least one neighbor of $v_r$ between $v_{r-1}$ and $v_{r+1}$
	that lies outside of $T'$.
	Let $u$ be such a neighbor of $v_r$ and let $T_u$ be a homothet of $T$ that contains $v_r$ and $u$ and no other point from $\mathcal{S}$.
	Note that $u$ is higher than $v_r$, thus $v_{r}u$ crosses either the right or the left side of $T'$.
	Suppose without loss of generality that $v_{r}u$ crosses the right side of $T'$ at a point $q_u$ (refer to Figure~\ref{fig:triangle}).
	\begin{figure}
		\centering
		\includegraphics[width=6cm]{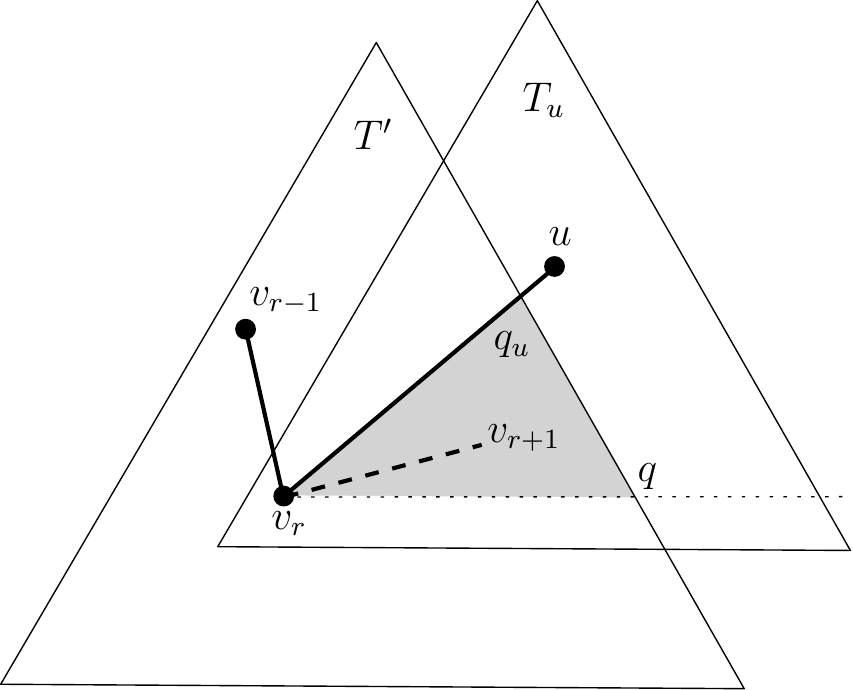}
		\caption{An illustration for the proof of Lemma~\ref{lem:triangle}.}
		\label{fig:triangle}
	\end{figure}
	It follows that the right side of $T_u$ is to the right of the right side of $T'$.
	Thus, a horizontal ray that begins at $v_r$ and goes to the right will first cross the right side of $T'$ (denote this crossing point by $q$)
	and then cross the right side of $T_u$
	(note that this ray does not cross the left sides of $T_u$ and $T'$ since $v_r \in T_u \cap T'$).
	Now consider the triangle $\triangle q_uqv_r$.
	All of its vertices are in $T_u \cap T'$, therefore $\triangle q_uqv_r \in T_u \cap T'$.
	However, since $v_{r+1}$ follows $u$ in the rotation of $v_r$,
	it follows that the edge $v_rv_{r+1}$ lies in $\triangle q_uqv_r$ since it cannot cross none of its sides.
	This is impossible since $v_{r+1}$ should be outside of $T_u$ and hence outside of $\triangle q_uqv_r$.
	Lemma \ref{lem:triangle} is proved.

\subsection{Parallelograms are universally good}
\label{subsec:squares}

In this section we prove Lemma~\ref{lem:parallelogram}.
Let $Q$ be a parallelogram, let $\mathcal{S}$ be a set of points in very general position with respect to $Q$,
and let $\DT:=\DT(Q,\mathcal{S})$ be the generalized Delaunay triangulation of $\mathcal{S}$ with respect to $Q$
such that $\DT$ is nice (i.e., the boundary of its outer face is a convex polygon).
By applying an affine transformation, we may assume without loss of generality that $Q$ is an axis-parallel square.
Since $\mathcal{S}$ is in very general position, no two points in $\mathcal{S}$ share the same $x$- or $y$-coordinate.

Suppose that $Q'$ is a homothet of $Q$ that contains at least $22$ points from $\mathcal{S}$ and that $\DT[Q']$ is a tree.
We will show that $Q'$ contains a good $3$-path.

Let $q \in \mathcal{S}$ be a point.
We partition the points of the plane into four open quadrants according to their position with respect to $q$:
$\NE(q)$ (North-East), $\NW(q)$ (North-West), $\SE(q)$ (South-East), and $\SW(q)$ (South-West).

 \begin{proposition}
 \label{prop:same-quadrant}
 Let $x,y,z$ be three points in $\mathcal{S}$ such that $xy$ and $xz$ are edges in $\DT$.
 Then for every quadrant ${\rm Qd} \in \{\NW,\NE,\SW,\SE\}$ if $y \in {\rm Qd}(x)$, then $z \notin {\rm Qd}(y)$.
 \end{proposition}

 \begin{proof}
 Suppose for contradiction and without loss of generality that $y \in \NE(x)$ and $z \in \NE(y)$.
 Then the smallest rectangle that contains $x$ and $z$ has $x$ at its bottom-left corner, $z$ at its top-right corner and $y$ in its interior.
 Therefore, there is no square that contains $x$ and $z$ and does not contain $y$ and so $xz$ cannot be an edge in $\DT$.
 \end{proof}


\begin{proposition}
\label{prop:diff-quadrant}
For every point $q \in \mathcal{S} \cap Q'$
there are no two neighbors of $q$ in $\DT[Q']$ that lie in the same quadrant of $q$.
\end{proposition}

 \begin{proof}
 Suppose for contradiction that $q$ has two neighbors, $x$ and $y$, that lie in the same quadrant. In this case we can choose them such that there is no other neighbor of $q$ between them.
 Assume without loss of generality that $x,y \in \NE(q)$,
 such that $qx$ forms a smaller angle with the $x$-axis than $qy$ (refer to Figure~\ref{fig:diff-quadrant}) and they follow each other immediately in the rotation of $q$.
 \begin{figure}
     \centering
     \includegraphics[width=3cm]{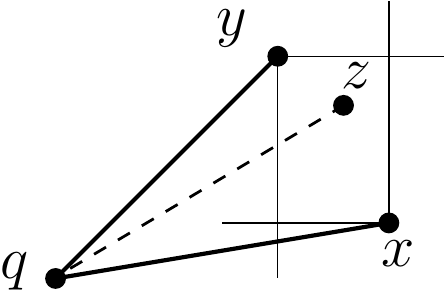}
 	\caption{An illustration for the proof of Proposition~\ref{prop:diff-quadrant}.}
 	\label{fig:diff-quadrant}
 \end{figure}
 It follows from Proposition~\ref{prop:consecutive} that there is a point $z \notin Q'$ such that $z$
 is a neighbor of $q$ in $\DT$ and is between $x$ and $y$ in the rotation of $q$.
 By Proposition~\ref{prop:same-quadrant} we have $y \notin \NE(x)$.
 Since $qx$ forms a smaller angle with the $x$-axis than $qy$ we have $y \notin \SE(x)$.
 If $y \in \SW(x)$, then $x \in \NE(y)$ which is impossible by Proposition~\ref{prop:same-quadrant}.
 Thus, $y \in \NW(x)$.
 Using the same arguments we get that $z \in \NW(x) \cap \SE(y)$.
 However, this implies that $z$ is contained in any axis-parallel rectangle that contains $x$ and $y$ and thus $z \in Q'$, a contradiction.
 \end{proof}


\begin{proposition}
\label{prop:cross-side}
Let $x$ and $y$ be two neighbors of $q$ in $\DT[Q']$.
Let $z \notin Q'$ be a neighbor of $q$ in $\DT$ that lies between $x$ and $y$ in the rotation of $q$
and let $Q_z$ be a square that contains $q$ and $z$ and no other point from $\mathcal{S}$.
Then:
\begin{packed_item}
\item if $x \in \NW(q)$ and $y \in \NE(q)$, then $qz$ crosses the top side of $Q'$, $x$ is to the left of $Q_z$ and $y$ is to the right of $Q_z$;
\item if $x \in \NE(q)$ and $y \in \SE(q)$, then $qz$ crosses the right side of $Q'$, $x$ is above $Q_z$ and $y$ is below $Q_z$;
\item if $x \in \SE(q)$ and $y \in \SW(q)$, then $qz$ crosses the bottom side of $Q'$, $x$ is to the right of $Q_z$ and $y$ is to the left of $Q_z$; and
\item if $x \in \SW(q)$ and $y \in \NW(q)$, then $qz$ crosses the left side of $Q'$, $x$ is below $Q_z$ and $y$ is above $Q_z$.
\end{packed_item}
\end{proposition}

\begin{proof}
	By symmetry it is enough to consider the first case, that is, $x \in \NW(q)$ and $y \in \NE(q)$.
	Since $z$ is between $x$ and $y$ in the rotation of $q$ we have $z \notin \SW(x)$ and  $z \notin \SE(y)$.
	By Proposition~\ref{prop:same-quadrant} $z \notin \NW(x)$ and  $z \notin \NE(y)$.
	Thus $z$ is to the right of $x$ and to the left of $y$.
	It follows that $z$ is above $Q'$ and  $qz$ crosses the top side of $Q'$. Therefore, the top side of $Q_z$ is above $Q'$.
	Thus, $qx$ cannot cross the top side of $Q_z$ so it must cross its left side.
	This implies that $x$ lies to the left of $Q_z$.
	Similarly, $qy$ cannot cross the top side of $Q_z$ so it must cross its right side.
	This implies that $y$ lies to the right of $Q_z$.
\end{proof}

Call a (simple) path in $\DT$ \emph{$x$-monotone} (resp., $y$-monotone) if there is no vertical (resp., horizontal) line that intersects the path in more than one point.

\begin{proposition}
\label{prop:xymonotone}
Every path in $\DT[Q']$ is $x$-monotone or $y$-monotone.
\end{proposition}

\begin{proof}
Suppose for contradiction that there is a path $p:=q_1 \mhyphen q_2 \mhyphen \ldots \mhyphen q_k$ which is neither $x$-monotone nor $y$-monotone.
Since $p$ is a polygonal path, it follows that there are two points, $q_i$ and $q_j$,
that are ``witnesses'' to the non-$x$- and non-$y$-monotonicity of $p$, respectively.
That is, both $q_{i-i}$ and $q_{i+1}$ are to the left of $q_i$ or both of them are to its right,
and both $q_{j-1}$ and $q_{j+1}$ are above $q_j$ or both of them are below $q_j$.
We choose $i$ and $j$ such that $|i-j|$ is minimized, and assume without loss of generality that $i<j$
(note that it follows from Proposition~\ref{prop:diff-quadrant} that $i \neq j$).
Thus, the sub-path $p':=q_i \mhyphen q_{i+1} \mhyphen \ldots,q_{j-1} \mhyphen q_j$ is both $x$-monotone and $y$-monotone.

By reflecting about the $x$- and/or $y$-axis if needed, we may assume that $p'$ is ascending, that is,
for every $l=i,\ldots,j-1$ we have $q_{l+1} \in \NE(q_l)$.
Then it follows from Proposition~\ref{prop:diff-quadrant} that $q_{i-1} \in \SE(q_i)$ and $q_{j+1} \in \SE(q_j)$.
By applying Proposition~\ref{prop:consecutive} to two vertices that are between $q_{i-1}$ and $q_{i+1}$ and follow each other immediately in the rotation of $q_i$ in $\DT[Q']$ (they can coincide with $q_{i-1}$ and/or $q_{i+1}$), we get that there is a point $x \notin Q'$ which is a neighbor of $q_i$ and is between $q_{i-1}$ and $q_{i+1}$ in the rotation of $q_i$,
and it follows from Proposition~\ref{prop:cross-side} that $q_ix$ crosses the right side of $Q'$.
The same argument implies that there is a point $y \notin Q'$ which is a neighbor of $q_j$ and is between $q_{j+1}$ and $q_{j-1}$ in the rotation of $q_j$,
such that $q_jy$ crosses the bottom side of $Q'$.
However, since $q_j$ is to the right of $q_i$ and above it, the edges $q_ix$ and $q_jy$ must cross, which is impossible.
\end{proof}

Call a $2$-path $\twopath{w}{q}{z}$ \emph{bad} if it is not good, that is,
there is an axis-parallel square $Q''$ that contains $q$, does not contain $w$ and $z$,
and $qw$ and $qz$ are edges in $\DT$ that cross the same side of $Q''$.
We say that $\twopath{w}{q}{z}$ is a bad \emph{left} $2$-path if $qw$ and $qz$ cross the left side of $Q''$,
and define \emph{right}, \emph{top}, and \emph{bottom} bad $2$-paths analogously.

\begin{proposition}
\label{prop:bad2paths}
Let $\twopath{w}{q}{z}$ be a $2$-path. Then:
\begin{packed_item}
\item $\twopath{w}{q}{z}$ is a bad left $2$-path iff $w \in \SW(q)$ and $z \in\NW(q)$, or vice versa;
\item $\twopath{w}{q}{z}$ is a bad right $2$-path iff $w \in \SE(q)$ and $z \in\NE(q)$, or vice versa;
\item $\twopath{w}{q}{z}$ is a bad top $2$-path iff $w \in \NW(q)$ and $z \in\NE(q)$, or vice versa; and
\item $\twopath{w}{q}{z}$ is a bad bottom $2$-path iff $w \in \SW(q)$ and $z \in\SE(q)$, or vice versa.
\end{packed_item}
\end{proposition}

\begin{proof}
	By symmetry it is enough to consider the first claim.
	If $\twopath{w}{q}{z}$ is a bad left $2$-path, then there is a square $Q''$ that separates it
	such that the edges $qw$ and $qz$ cross the left side of $Q''$.
	Therefore, these edges go leftwards from $q$ and so $w,z \in \SW(q) \cup \NW(q)$.
	It follows from Proposition~\ref{prop:diff-quadrant} that $w \in \SW(q)$ and $z \in\NW(q)$, or vice versa.
	
	For the other direction, assume without loss of generality that $w \in \SW(q)$ and $z \in\NW(q)$.
	Let $Q''$ be the square whose left side is the straight-line segment between $((q)_x-\varepsilon,(w)_y)$ and $((q)_x-\varepsilon,(z)_y)$,
	for some small $\varepsilon > 0$.
	Then $Q''$ separates $\twopath{w}{q}{z}$ and both $qw$ and $qz$ cross its left side,
	therefore, $\twopath{w}{q}{z}$ is a bad left $2$-path.
\end{proof}

\begin{proposition}
\label{prop:4-bads}
Every path in $\DT[Q']$ contains at most four bad $2$-paths.
\end{proposition}

\begin{proof}
Let $p:=q_1 \mhyphen q_2 \mhyphen \ldots \mhyphen q_{k}$ be a simple path in $\DT[Q']$
and suppose for a contradiction that $p$ contains at least five bad $2$-paths.
By Proposition~\ref{prop:xymonotone} the path $p$ is $x$-monotone or $y$-monotone.
Assume without loss of generality that $p$ is $y$-monotone and that it goes upwards, that is, $q_{i+1}$ is above $q_i$ for every $i=1,2,\ldots,k-1$.
It follows that $p$ does not contain bad top or bad bottom $2$-paths, for otherwise it would not be $y$-monotone.
It is not hard to see that bad left and bad right $2$-paths must alternate along $p$,
that is, between every two bad left $2$-paths there is a bad right $2$-path and vice versa.

Consider the first five such bad $2$-paths along the path $p$, and denote them by
$\twopath{q_{i_1-1}}{q_{i_1}}{q_{i_1+1}}$,  $\twopath{q_{i_2-1}}{q_{i_2}}{q_{i_2+1}}$,  $\twopath{q_{i_3-1}}{q_{i_3}}{q_{i_3+1}}$,
$\twopath{q_{i_4-1}}{q_{i_4}}{q_{i_4+1}}$ and $\twopath{q_{i_5-1}}{q_{i_5}}{q_{i_5+1}}$.
By symmetry we may assume without loss of generality that $\twopath{q_{i_1-1}}{q_{i_1}}{q_{i_1+1}}$ is a bad left $2$-path,
and therefore $\twopath{q_{i_3-1}}{q_{i_3}}{q_{i_3+1}}$ and  $\twopath{q_{i_5-1}}{q_{i_5}}{q_{i_5+1}}$ are also bad left $2$-paths, whereas
the $2$-paths $\twopath{q_{i_2-1}}{q_{i_2}}{q_{i_2+1}}$ and $\twopath{q_{i_4-1}}{q_{i_4}}{q_{i_4+1}}$ are bad right.

Note that we may assume without loss of generality that $q_{i_1}$ is to the right of $q_{i_4}$,
for otherwise $q_{i_5}$ must be to the right of $q_{i_2}$ and by reflecting about the $x$-axis
and renaming the points we get the desired assumption.
By applying Proposition~\ref{prop:consecutive} to two vertices that are between $q_{i_1-1}$ and $q_{i_1+1}$ and follow each other immediately in the rotation of $q_{i_1}$ in $\DT[Q']$ (they can coincide with $q_{i_1-1}$ and/or $q_{i_1+1}$), we get that $q_{i_1}$ has a neighbor $z \notin Q'$ between $q_{i_1-1}$ and $q_{i_1+1}$ in the rotation of $q_{i_1}$.
Let $Q_z$ be a square that contains $q_{i_1}$ and $z$ and no other point from $\mathcal{S}$ and let $s_z$ be its side length
(refer to Figure~\ref{fig:4-bads}).
\begin{figure}
    \centering
    \includegraphics[width=5cm]{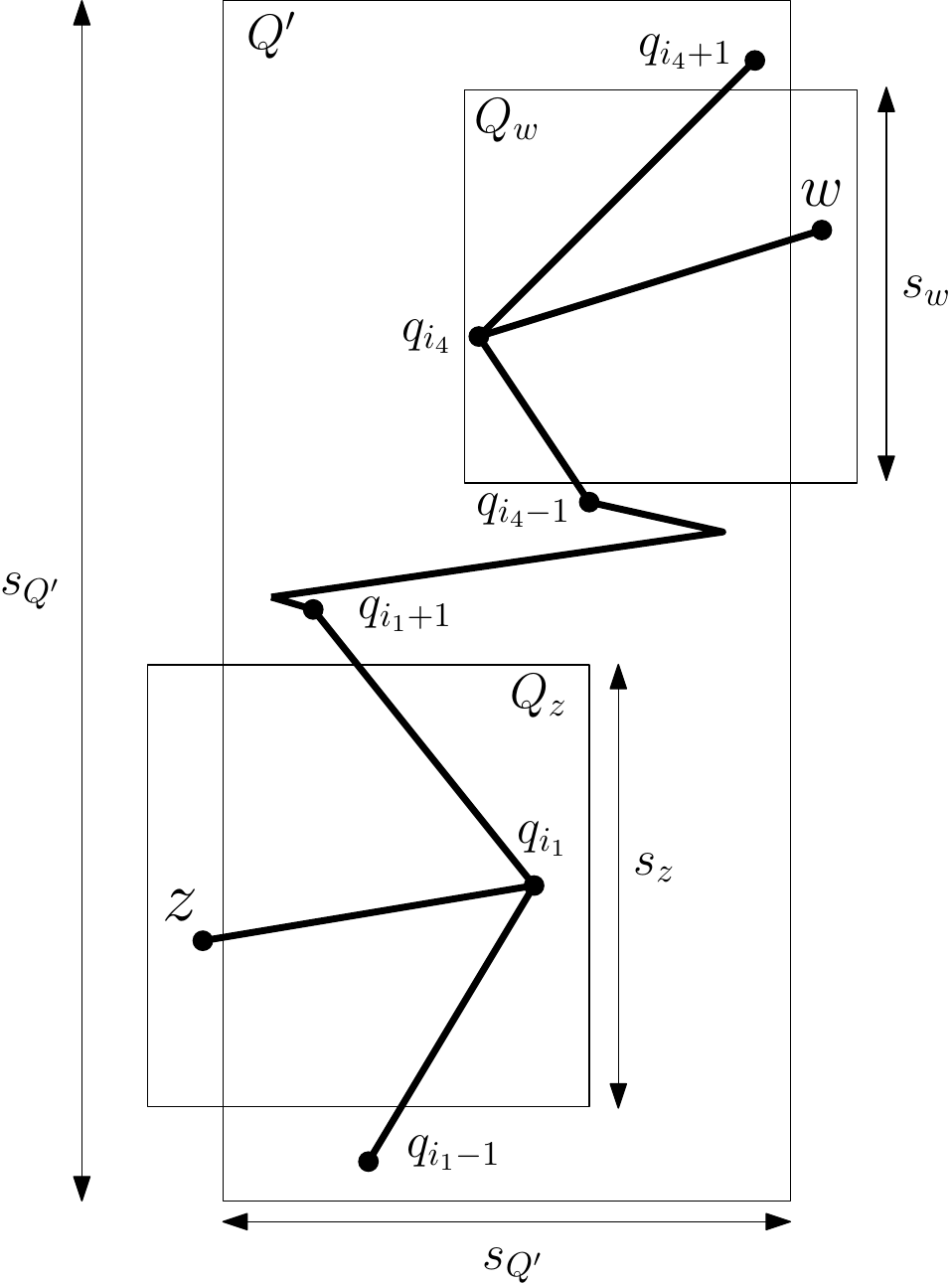}
	\caption{An illustration for the proof of Proposition~\ref{prop:4-bads}.}
	\label{fig:4-bads}
\end{figure}
It follows from Proposition~\ref{prop:cross-side} that $q_{i_1-1}$ lies below $Q_z$, $q_{i_1+1}$ lies above $Q_z$, and $z$ lies to the left of $Q'$.
Therefore, $(q_{i_1+1})_y - (q_{i_1-1})_y > s_z$.
Similarly, $q_{i_4}$ has a neighbor $w \notin Q'$ between $q_{i_4+1}$ and $q_{i_4-1}$ in the rotation of $q_{i_4}$.
Let $Q_w$ be a square that contains $q_{i_4}$ and $w$ and no other point from $\mathcal{S}$ and let $s_w$ be its side length.
Then $q_{i_4-1}$ lies below $Q_w$, $q_{i_4+1}$ lies above $Q_w$, and $w$ lies to the right of $Q'$.
Therefore, $(q_{i_4+1})_y - (q_{i_4-1})_y > s_w$.

Note that since $q_{i_1}$ is to the right of $q_{i_4}$ and $z$ and $w$ are to the left and to the right of $Q'$, respectively,
we have $s_z + s_w > ((q_{i_1})_x - (z)_x) + ((w)_x - (q_{i_4})_x) > s_{Q'}$, where $s_{Q'}$ is the side length of $Q'$.
Observe also that since there are at least two other vertices between $q_{i_1}$ and $q_{i_4}$ along $p$,
we have that $q_{i_1+1} \neq q_{i_4-1}$, and thus $q_{i_1+1}$ lies below $q_{i_4-1}$.
This implies that $((q_{i_1+1})_y - (q_{i_1-1})_y) + ((q_{i_4+1})_y - (q_{i_4-1})_y) < s_{Q'}$.
Combining the inequalities we get, $s_{Q'} > ((q_{i_1+1})_y - (q_{i_1-1})_y) + ((q_{i_4+1})_y - (q_{i_4-1})_y) > s_z + s_w > ((q_{i_1})_x - (z)_x) + ((w)_x - (q_{i_4})_x) > s_{Q'}$,
a contradiction.
\end{proof}

To complete the proof of Lemma~\ref{lem:parallelogram} we will consider a path of length $11$ in $\DT[Q']$.
It follows from Proposition~\ref{prop:diff-quadrant} that for every $q \in \mathcal{S} \cap Q'$ we have $\deg_{\DT[Q']}(q) \leq 4$.
This implies that if $Q'$ contains at least $1+\sum_{i=1}^5 4^i=1366$ points from $\mathcal{S}$, then $\DT[Q']$
contains a simple path of length at least $11$.
However, one can show that already $22$ points suffice to guarantee the existence of a path of length $11$. To prove this, we will need the following proposition. 

\begin{proposition}
	\label{prop:degree-2}
	There are at most two points in $\mathcal{S} \cap Q'$ whose degree in $\DT[Q']$ is greater than two. If one of these points has degree four, then no other point has degree greater than two.
\end{proposition}

\begin{proof}
	Suppose first that there is a point $q_k$ with $\deg_{\DT[Q']}(q_k)=4$ and another point $q_0$ with $\deg_{\DT[Q']}(q_0)\ge 3$ (see Figure~\ref{fig:large-degrees}(a) for an example). Let $p:=q_0 \mhyphen q_1 \mhyphen q_2 \mhyphen \ldots \mhyphen q_k$ be the path connecting these points in the tree $\DT[Q'$]. We may assume without loss of generality that $q_1 \in \SE(q_0)$ and that $q_0$ has a neighbor $z_0 \in \SW(q_0)$.
	In this case the path $z_0 \mhyphen q_0 \mhyphen q_1 \mhyphen q_2 \mhyphen \ldots \mhyphen q_k$ is not $y$-monotone. 
	Therefore this path must be $x$-monotone by Proposition~\ref{prop:xymonotone}, and thus $p$ is also $x$-monotone.
	It follows that $q_{k-1} \in \SW(q_k) \cup \NW(q_k)$. 
	Since by Proposition~\ref{prop:same-quadrant} a point cannot have two neighbors in the same quadrant and the degree of $q_k$ is four, 
	it has another neighbor $z_k \ne q_{k-1}$ in $\SW(q_k) \cup \NW(q_k)$.
	Therefore, the path $z_0 \mhyphen q_0 \mhyphen q_1 \mhyphen \ldots \mhyphen q_k \mhyphen z_k$ is not monotone, a contradiction.
	
	\begin{figure}
		\centering
		\subfloat[The degree of $q_0$ is at least three and the degree of $q_k$ is four. The path $z_0 \mhyphen q_0 \mhyphen q_1 \mhyphen \ldots \mhyphen  q_k \mhyphen z_k$ is not monotone.]{\includegraphics[width= 7cm]{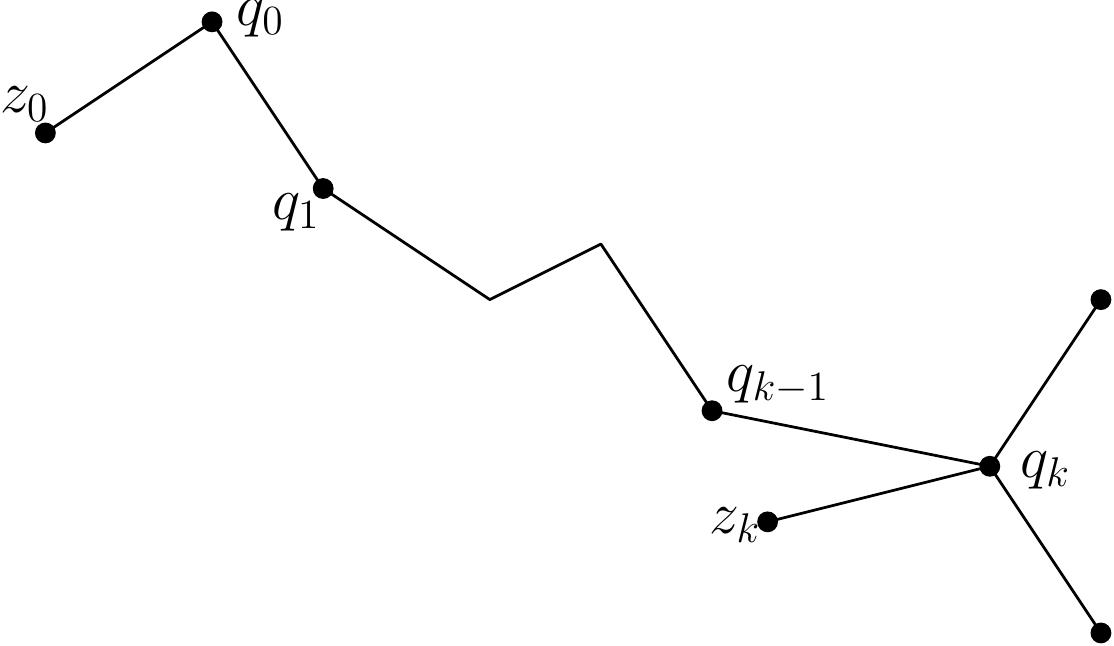}}
		\hspace{1cm}
		\subfloat[The degree of $q_0$, $q_j$ and $q_k$ is three. We can assume that $q_0$ has a neighbor $z_0 \in \SW(q_0)$. In case none of $q_j$ and $q_k$ has a neighbor in both of their $\NW$ and $\SW$ quadrants, there is a path $y_j \mhyphen  q_j \mhyphen \ldots \mhyphen  q_k \mhyphen y_k$ which is not monotone.]{\includegraphics[width= 7cm]{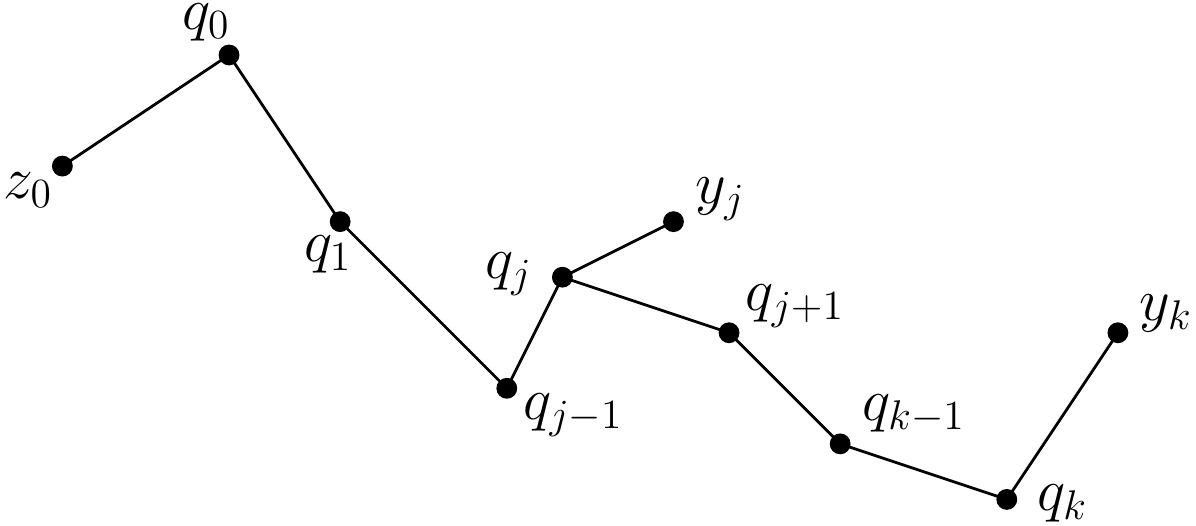}}
		\caption{Illustrations for the proof of Proposition~\ref{prop:degree-2}.}
		\label{fig:large-degrees}		
	\end{figure}
	
	\medskip
	Now suppose that each point in $\DT[Q']$ has degree at most three, and suppose for contradiction that there are at least three points in $\mathcal{S} \cap Q'$ whose degree in $\DT[Q']$ is three.	Since $\DT[Q']$ is a tree, there must exist a path $p:=q_0 \mhyphen q_1 \mhyphen q_2 \mhyphen \ldots \mhyphen q_k$ ($k\ge 2)$ between two points $q_0$ and $q_k$ with degree three that contains a third point $q_j$ ($0 < j < k$) with degree three. 
	We can assume without loss of generality that 
	$q_1\in \SE(q_0)$.
	Since the degree of $q_0$ is three, it has a neighbor in $\NE(q_0) \cup \SW(q_0)$.

	We can also assume that $q_0$ has a neighbor $z_0 \in \SW(q_0)$ (see Figure~\ref{fig:large-degrees}(b) for an example).	
	In this case the path $z_0 \mhyphen q_0 \mhyphen q_1 \mhyphen \ldots \mhyphen q_k$ is not $y$-monotone, thus it must be $x$-monotone and so $p$ is also $x$-monotone.
	It follows, that if $q_i$, for $i \in \{j,k\}$, has two neighbors in $\SW(q_i) \cup \NW(q_i)$, then for one of them, denote it by $z_i \ne q_{i-1}$, a path that ends with  $\twopath{q_{i-1}}{q_i}{z_i}$ is not $x$-monotone and hence
	the path $z_0 \mhyphen q_0 \mhyphen q_1 \mhyphen \ldots \mhyphen  q_{i-1} \mhyphen q_i \mhyphen z_i$ is not monotone.
	Therefore, $q_i$ has two neighbors in $\SE(q_i) \cup \NE(q_i)$, for $i \in \{j,k\}$.
	It follows that $q_j$ has a neighbor $y_j \ne q_{j+1}$ such that a path that starts with $\twopath{y_j}{q_j}{q_{j+1}}$ is not $x$-monotone, and $q_k$ has a neighbor $y_k$ such that a path that ends with $\twopath{q_{k-1}}{q_k}{y_k}$ is not $y$-monotone.
	Therefore the path $y_j \mhyphen q_j \mhyphen q_{j+1} \mhyphen \ldots \mhyphen  q_{k-1} \mhyphen q_k \mhyphen y_k$ is not monotone, a contradiction.
\end{proof}

\begin{lemma}
If $Q'$ contains at least $22$ points from $\mathcal{S}$, then $\DT[Q']$ contains a simple path of length at least $11$.
\end{lemma}

\begin{proof}		
If $\DT[Q']$ contains a vertex $q$ whose degree is four, then it follows from Proposition~\ref{prop:degree-2} that by deleting it we decompose the tree $\DT[Q']$ into four paths, having $21$ vertices altogether. By the pigeonhole principle either the first two or the second two together have at least $11$ vertices. Along with $q$ they form a path with $12$ vertices, as required.
	
	If $\DT[Q']$ does not contains a vertex whose degree is four, then
	by Proposition~\ref{prop:degree-2} it contains at most two vertices with degree three, and no other vertex has degree greater than two. We can assume that we have exactly two  vertices whose degree is three, $q$ and $q'$, and let $p$ be the path connecting them in $\DT[Q']$. By deleting these two vertices we obtain five paths: a path that consists of $p$ without its two endpoints, two paths $p_1,p_2$ that are incident to $q$ and two paths $p_3,p_4$ that are incident to $q'$. Considering the number of vertices in each of these paths we have $|V(p_1)|+|V(p_2)|+|V(p)|+|V(p_3)|+|V(p_4)|+|V(p)| \ge 22+2=24$, since $|V(p)| \ge 2$. Therefore, one of the paths formed by $p_1,p,p_3$ and $p_2,p,p_4$ must contain at least $12$ vertices, as required.
\end{proof}

Let $p:=q_1 \mhyphen q_2 \mhyphen \ldots \mhyphen q_{12}$ be a simple path of length $11$ in $\DT[Q']$.
By Proposition~\ref{prop:4-bads} there are at most four bad $2$-paths $\twopath{q_{i-1}}{q_{i}}{q_{i+1}}$ in $p$.
Therefore, there is $2 \leq i \leq 10$ such that $\twopath{q_{i-1}}{q_{i}}{q_{i+1}}$ and $\twopath{q_{i}}{q_{i+1}}{q_{i+2}}$ are good $2$-paths,
and therefore $Q'$ contains a good $3$-path $\threepath{q_{i-1}}{q_{i}}{q_{i+1}}{q_{i+2}}$.
Lemma~\ref{lem:parallelogram} is proved.\hfill$\qed$

\subsection{A universally good polygon is either a triangle or a parallelogram}
\label{sec:construction}

%

In this section we prove that triangles and parallelograms are the only universally good polygons. That is, for any other polygon $P$ we can construct a set of points $\mathcal{S}$ such that there is a homothet of $P$ that intersects $\DT(P,\mathcal{S})$ in a long path, while every other vertex of this path can be separated from its neighbors by the same side of a homothet of $P$ (thus there is no good $3$-path in $P$). See Figure \ref{fig:construction}. The rest of this section contains the exact description and validity of this construction.

	\begin{figure}
	\centering
	\includegraphics[width=9cm]{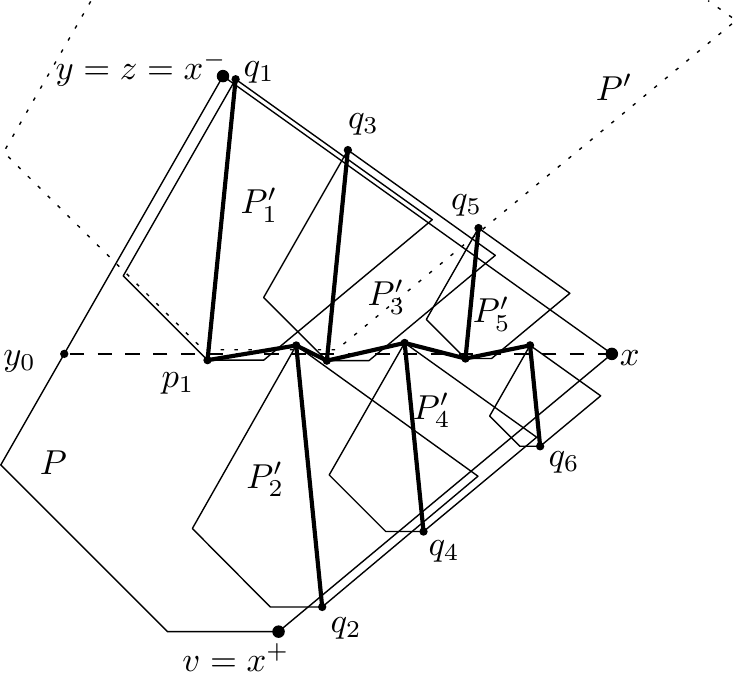}
	\caption{A construction showing that $P$ is not universally good.}
	\label{fig:construction}
\end{figure}

We start with a simple statement that will be used later.

\begin{proposition} \label{prop:inner-homothet}
	Let $P'$ be a homothet of a convex closed polygon $P$ that is contained in $P$.
	If $\partial P \cap \partial P' \neq \emptyset$, then an edge of $P'$ is contained in an edge of $P$ (to which it is homothetic).
\end{proposition}

\begin{proof}
	Suppose for contradiction that $\partial P \cap \partial P' \neq \emptyset$ and no edge of $P'$ is contained in an edge of $P$.
	Then there is a vertex $v'$ of $P'$ that lies on $\partial P$ such that the edges that are incident to $v'$ lie in the interior of $P$.
	It is impossible thus that $v'$ coincides with $v$, the vertex to which it is homothet, since then these vertices would be incident to edges of different slopes.
	$v'$ cannot coincide with another vertex of $P$ either, since then there is a direction in which $v'$ is extreme in $P'$ and a different vertex than $v$ is extreme at $P$.
	Thus $v$ lies on an edge of $P$.
	One of the endpoints of this edge must be $v$, for otherwise as before there is a direction in which $v'$ is extreme in $P'$ and a different vertex than $v$ is extreme at $P$.
	But then, again, it follows that $v$ and $v'$ are incident to edges of different slopes.
	Therefore, there is an edge of $P'$ which is contained in an edge of $P$, and since both $P'$ and $P$ are on the same side
	of the line through this edge, it follows that these edges are homothetic to each other.
\end{proof}


Suppose that $P$ is a convex closed $n$-gon which is neither a triangle nor a parallelogram.
We will show that $P$ is not universally good.
%
%
%
	Let $uv$ be an edge of $P$ which is not parallel to any other edge of $P$ if $n=4$ and an arbitrary edge otherwise.
	Assume without loss of generality that $uv$ is horizontal, $u$ is left of $v$, and $P$ lies above $uv$.
	
	In order to have only a single case to deal with, if $P$ has a unique highest vertex, then we regard this vertex as two vertices joined by a horizontal edge of length $0$.	This way $P$ always has a top horizontal edge $yz$ (possibly of length $0$) such that $y$ is left of $z$ and $P$ lies below $yz$. 
	Since $P$ is neither a triangle nor a parallelogram it has a vertex $x \notin \{u,v,y,z\}$.
	Assume without loss of generality that $x$ is on the clockwise polygonal chain from $y$ to $v$ on the boundary of $P$.
	Denote by $y_0$ the point on the clockwise polygonal chain from $u$ to $y$ on the boundary of $P$ that has the same $y$-coordinate as $x$
	and observe that the line-segment $y_0x$ lies in $P$.
	
	\begin{proposition}\label{prop:D-edge}
		For every two distinct points $a$ and $b$ that lie on $y_0x$ there is a homothet of $P$, denote it by $P'$,
		such that $P'$ is contained in $P$ and $y_0x \cap P' = ab$.
	\end{proposition}
	
	\begin{proof}
		We can obtain such a homothet $P'$ as follows. Initially, set $P':=P$.
		Now shrink $P'$ with respect to $a$ until $b$ lies on $\partial P'$.
		Then shrink $P'$ with respect to $b$ until $a$ lies on $\partial P'$.
		Since $P'$ is convex we have $y_0x \cap P' = ab$.
	\end{proof}

	For a polygon $R$ and a vertex $r \in R$, we denote by $r^-$ and $r^+$ the vertices of $R$ that precede and succeed $r$, respectively, in the clockwise order of the vertices of $R$.
	For a homothet $P'$ (resp., $P_i$) of $P$ and a vertex $r \in P$, we denote by $r'$ (resp., $r_i$) the vertex of $P'$ (resp., $P_i$)
	that is homothetic to $r$. Since, e.g., $(r_i)^+ = (r^+)_i$, we simply write $r^+_i$, $r^-_i$, $r'^-$, etc.
	
	Given an integer $k \geq 1$, we describe a way to construct a set of $k$ homothets of $P$, $P_1,P_2,\ldots,P_k$, such that:
	\begin{packed_enum}
		\item[(1)] $P$ contains $P_i$, for every $i \geq 1$;
		\item[(2o)] the edge $x^-_ix_i$ of $P_i$ is contained in the edge $x^-x$ of $P$, for every odd $i \geq 1$;
		\item[(2e)] the edge $x_ix^+_i$ of $P_i$ is contained in the edge $xx^+$ of $P$, for every even $i \geq 2$;
		\item[(3o)] the edge $u_iv_i$ of $P_i$ is contained in the open segment $y_{i-1}x$, for every odd $i \geq 1$; and
		\item[(3e)] the edge $y_iz_i$ of $P_i$ is contained in the open segment $v_{i-1}x$, for every even $i \geq 2$.
	\end{packed_enum}
	
	We will use the following proposition for the construction.
	
	\begin{proposition} \label{prop:add-homothet}
		Let $u^*$ be a point in the open line-segment $y_0x$.
		Then there is homothet of $P$, $P'$, such that $P'$ is contained in $P$,
		$u'v'$ is contained in the open line-segment $u^*x$,
		and $x'^-x'$ is contained in $x^-x$.
	\end{proposition}

	\begin{proof}
		For an arbitrary $u'$ on $u^*x$ we can fix a homothet of $P$, denote it by $P''$, such that its edge $u''v''$ that is homothetic to $uv$ coincides with $u'x$ (i.e., $u''=u'$ and $v''=x$). Clearly, by choosing $u'$ close enough to $x$, the intersection of $\partial P''$ and $\partial P$ is contained in the edge $x^-x$.
		Since $P$ is convex, it follows that the vector $\overrightarrow{vv^-}$ forms a smaller angle with the positive $x$-axis than does the vector $\overrightarrow{xx^-}$.
		Therefore, the (open) edge $v''{v^-}''$ of $P''$ homothetic to $vv^-$ lies outside of $P$.
		We now continuously shrink $P''$ with respect to $u'$ until it is contained in $P$.
		Let $P'$ be the resulting homothet of $P$ (see Figure~\ref{fig:add-homothet}).
		\begin{figure}
			\centering
			\includegraphics[width=5cm]{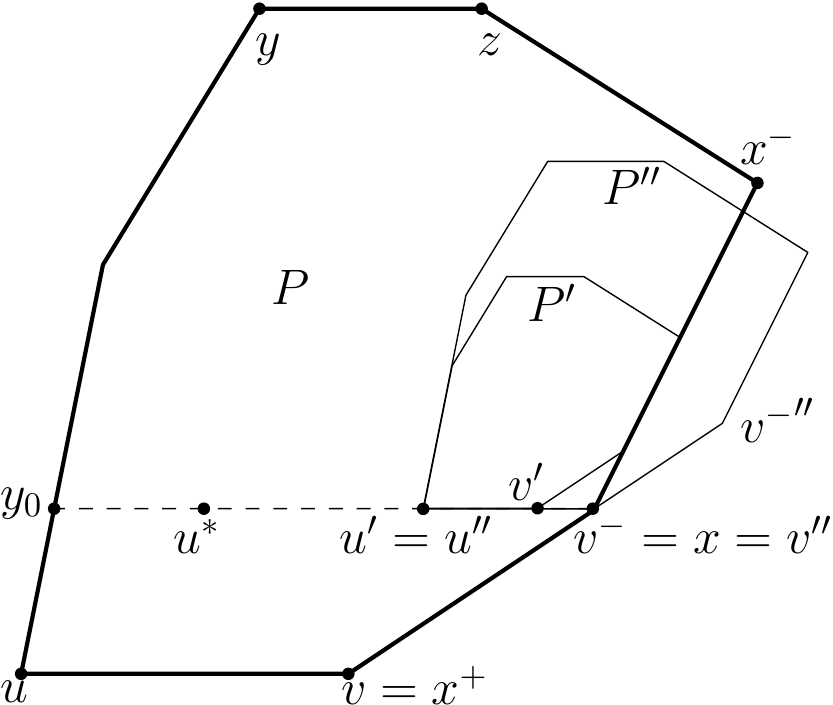}
			\caption{An illustration for the proof of Proposition~\ref{prop:add-homothet}.}
			\label{fig:add-homothet}
		\end{figure}
		Clearly, $P'$ is contained in $P$.
		Note that an edge of $P'$ is contained in an edge of $P$ by Proposition~\ref{prop:inner-homothet}. Also, as $P'$ is in $P''$ and the intersection of $\partial P''$ and $\partial P$ is contained in $x^-x$, the only possibility is that $x'^-x'$ is contained in  $x^-x$.
		Observe also that since we had to shrink $P''$, the vertex $v'$ of $P'$ that is homothetic to $v$ lies on the open segment $u'x$.
	\end{proof}
	
	By reflecting $P$ about the $x$-axis and applying Proposition \ref{prop:add-homothet} we get:
	
	\begin{corollary}\label{cor:add-homothet-down}
		Let $u^*$ be a point in the open line-segment $y_0x$.
		Then there is homothet of $P$, $P'$, such that $P'$ is contained in $P$,
		$y'z'$ is contained in the open line-segment $u^*x$,
		and $x'^+x'$ is contained in $x^+x$.
	\end{corollary}
	
	By applying Proposition~\ref{prop:add-homothet} with $u^*=y_0$ we get a homothet $P_1$ that satisfies Properties (1)--(3) such that $v_1$ is on the open segment $y_0x$.
	Suppose that we have homothets $P_1,\ldots,P_i$ that satisfy Properties (1)--(3). 	For an even $i$ we apply Proposition \ref{prop:add-homothet}  with $u^*=v_{i}$ to get $P_{i+1}$. For an odd $i$ we apply Corollary \ref{cor:add-homothet-down} with $u^*=z_{i}$ to get $P_{i+1}$. Thus we obtain homothets of $P$ that satisfy Properties (1)--(3). 	Note that $P_i$ is contained in the closed half-plane that is bounded from below (resp., above) by the line containing $y_0x$ for every odd (resp., even) $i \geq 1$.
	
	\medskip
	Suppose for contradiction that $P$ is universally good with a constant $c := c_g(P)$. We may assume that $c \equiv 0 \mod 4$ (by increasing $c$ if necessary).
	We will construct a set of points $\mathcal{S}$ such that $|\mathcal{S} \cap P|=c$,
	$\DT(P,\mathcal{S})[P]$ is a path, and $P$ does not contain a good $3$-path.
	This will contradict the fact that $P$ is universally good.
	
	We begin with an empty set of points $\mathcal{S}$ and a set of homothets of $P$, $P_1,\ldots,P_k$, as above, for $k=c$.
	Next, for every $i=1,\ldots,c-1$ we add a point $p_i$ to $\mathcal{S}$ as follows:
	if $i$ is odd, then $p_i=u_i$, that is, it coincides with the vertex of $P_i$ that is homothetic to $u$,
	whereas if $i$ is even, then $p_i=y_i$, that is, it coincides with the vertex of $P_i$ that is homothetic to $y$.
	

	\begin{proposition}\label{prop:same-order}
		The points $x^-_1,x^-_3,\ldots,x^-_{c-1}$ appear on $x^-x$ in this order.
		The points $x^-_2,x^-_4,\ldots,x^-_{c}$ appear on $x^+x$ in this order.
	\end{proposition}
	
	\begin{proof}
		By symmetry it is enough to prove the first statement.
		Suppose for contradiction that there are two odd indices $i_1 < i_2$ such that $x^-_{i_1}$ is closer than $x^-_{i_2}$ to $x$.
		Consider the triangles $\triangle u_{i_1}v_{i_1}x^-_{i_1}$ and $\triangle u_{i_2}v_{i_2}x^-_{i_2}$.
		Since $u_{i_2}v_{i_2}$ is to the right of $u_{i_1}v_{i_1}$ on $y_0x$, it follows that their boundaries cross at four points
		(see Figure~\ref{fig:crossing-triangles}).
		However, these triangles are homothetic, and thus such crossing is impossible.
			\end{proof}
		\begin{figure}
			\centering
			\includegraphics[width=7cm]{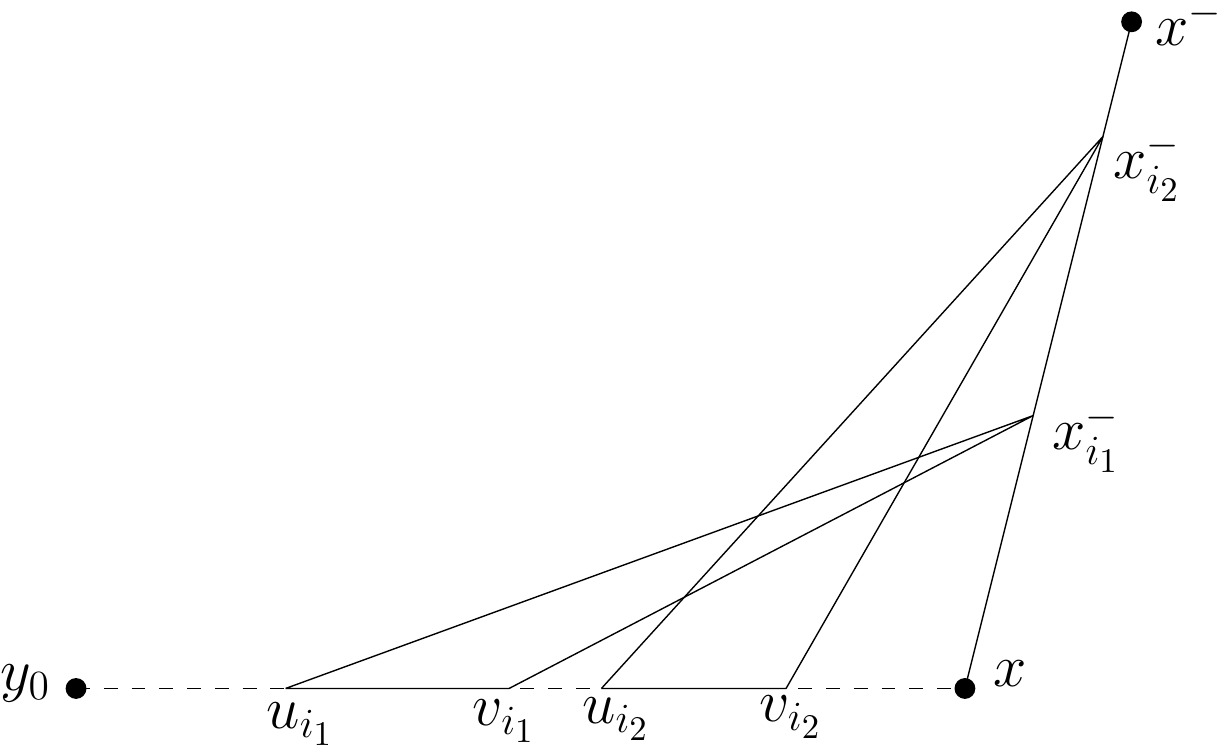}
			\caption{An illustration for the proof of Proposition~\ref{prop:same-order}.}
			\label{fig:crossing-triangles}
		\end{figure}

	It follows from Proposition~\ref{prop:same-order}
	that $x^-_i$ (resp., $x^+_i$) is not contained in $P_{i+2j}$ for every $i \geq 1$ and every $j$ for which $P_{i+2j}$ exists.

	Let $\varepsilon >0$ be some small positive constant that is much smaller than the smallest distance between any pair of distinct points at
	vertices of the above mentioned polygons (that is, $P$ and $P_1,\ldots,P_c$).
	For every odd $i$ we fix a point $q_i$ outside of $P$ at distance $\varepsilon i$ from $x^-_i$  in the direction of the normal to $x^-x$ and add this point to $\mathcal{S}$.
	Similarly, for every even $i$ we fix a point $q_i$ outside of $P$ at distance $\varepsilon i$ from $x^+_i$ in the direction of the normal to $xx^+$ and add this point to $\mathcal{S}$.
	Then, for every $i$, by slightly inflating $P_i$ with respect to some inner point,
	we obtain a homothet of $P$, denote it by $P'_i$, such that $P'_i \cap \mathcal{S} = \{p_i, q_i\}$.
	
	
	Let $\DT:=\DT(P,\mathcal{S})$ be the generalized Delaunay triangulation of $\mathcal{S}$ with respect to $P$.
	It follows from the construction that $(p_i,q_i)$ is an edge in $\DT$, for every $i=1,\ldots,c$.
	By Proposition~\ref{prop:D-edge} it follows that $\DT$ also contains the edge $(p_i,p_{i+1})$, for every $i=1,\ldots,c-1$.
	
	Next we apply a small perturbation of the points in $\mathcal{S}$ and slightly scale and translate the polygons $P_i$,
	such that for every $i=1,\ldots,c$ we have:
\begin{packed_enum}
	\item $P_i$ still contains the same (perturbed) points $p_i$ and $q_i$ and no other (perturbed) point from $\mathcal{S}$;
	\item if $i$ is odd, then the point $p_i$ lies slightly below $y_0x$; and
	\item if $i$ is even, then the point $p_i$ lies slightly above $y_0x$.
\end{packed_enum}	

	Thus, the above-mentioned edges of $\DT$ of type  $(p_i,q_i)$ and  $(p_i,p_{i+1})$ are still edges of $\DT$. See Figure~\ref{fig:construction} for an example of the construction at this point.

	To complete the construction we add some points to $\mathcal{S}$, as in the proof of Theorem~\ref{thm:general},
	to obtain a nice set of points with respect to $P$ and perturb the points to obtain a set of points in very general position with respect to $P$.
	
	Observe that $\DT[P]$ consists of the path $p_1 \mhyphen p_2 \mhyphen \ldots \mhyphen p_c$.
	Therefore, if $P$ contains a good $3$-path,
	then it must be of the form $\threepath{p_{i}}{p_{i+1}}{p_{i+2}}{p_{i+3}}$ for some $1 \leq i \leq c-3$.
	Thus, it is enough to prove that for every even $i$ the $2$-path $\twopath{p_{i-1}}{p_i}{p_{i+1}}$ is not good.
	Suppose for contradiction that $\twopath{p_{i-1}}{p_i}{p_{i+1}}$ is a good $2$-path for some even $i$. Recall that while we allowed $y=z$, we know that $u\ne v$, that is, $uv$ is an edge of positive length.
	Since $i$ is even, $p_i$ lies slightly above $y_0x$ whereas $p_{i-1}$ and $p_{i+1}$ lie slightly below $y_0x$.
	Let $P'$ be a homothet of $P$ such that the endpoints of its edge $u'v'$ that is homothetic to $uv$ are on $y_0x$,
	and $u'$ (resp., $v'$) has the same $x$-coordinate as $p_{i-1}$ (resp., $p_{i+1}$).
	It follows that $p_i$ lies inside $P'$ whereas $p_{i-1}$ and $p_{i+1}$ are outside of $P'$.
	Moreover, the edges $p_ip_{i-1}$ and $p_ip_{i+1}$ both cross the edge $u'v'$ of $P'_i$.
	Therefore, the $2$-path $\twopath{p_{i-1}}{p_i}{p_{i+1}}$ is not a good $2$-path, a contradiction.

\section{Discussion}
\label{sec:discussion}

The main open problem related to our work is the following.

\begin{problem}
\label{prob:convex}
Is it true that for every convex polygon $P$ there is a constant $m:=m(P)$ such that it is possible
to $2$-color any set of points $\mathcal{S}$ such that every homothet of $P$
that contains at least $m$ points from $\mathcal{S}$ contains points of both colors?
\end{problem}

By Theorem~\ref{thm:general} it would be enough to show that every convex polygon is universally good.
However Theorem \ref{thm:construction} shows that no other polygon is universally good besides triangles and parallelograms, thus for other classes of convex polygons additional ideas are needed. We remark that in a recent manuscript using the techniques developed in this article Keszegh and P\'alv\"olgyi \cite{K2016} solved the above problem with 3 colors.


	
We conclude with two challenging related open problems.
Considering coloring of points with respect to disks, recall that in \cite{unsplittable} it is proved that
there is no constant $m$ such that any set of points in the plane can be $2$-colored such that any (unit) disk
that contains at least $m$ points from the given set is non-monochromatic (that is, contains points of both colors).
Coloring the points with four colors such that any disk that contains at least two points is non-monochromatic is
easy since the (generalized) Delaunay graph is planar.
Therefore, it remains an interesting open problem whether there is a constant $m$ such that any set of points in the plane
can be $3$-colored such that any disk that contains at least $m$ points is non-monochromatic (this problem was posed originally in \cite{Kweak,Khalf}, for more general variants see also \cite{K2016}).

Perhaps the most interesting problem of coloring geometric hypergraphs
is to color a planar set of points $\mathcal{S}$ with the minimum possible number of colors,
such that every axis-parallel rectangle that contains at least two points from $\mathcal{S}$ is non-monochromatic.
It is known that $\Omega(\log (|\mathcal{S}|) / \log^2\log (|\mathcal{S}|))$ colors are sometimes needed~\cite{CPST09},
and it is conjectured that $polylog(|\mathcal{S}|)$ colors always suffice.
The latter holds when considering rectangles that contain at least three points~\cite{AP13},
however, for the original question only polynomial upper bounds are known~\cite{AEGR12,Ch12,HS05,PT03}.

\bigskip
\noindent \textbf{Acknowledgement.}
We thank the reviewers for reading our paper and for their valuable remarks. 
In particular, one reviewer's suggestions helped simplifying some basic lemmas and improving the constant in our main theorem.
\bibliography{squarecoloring}
\bibliographystyle{abbrv}

\end{document}